\documentclass[10pt]{amsart}

\usepackage{amssymb}
\usepackage[leqno]{amsmath}
\usepackage{mathrsfs}
\usepackage{stmaryrd}
\usepackage{chemarrow}
\usepackage[norelsize]{algorithm2e}
\usepackage{enumerate}
\usepackage{graphicx}
\usepackage[pdftex,bookmarksnumbered,bookmarksopen,colorlinks,linkcolor=red,anchorcolor=black,citecolor=blue,urlcolor=blue]{hyperref}
\usepackage[all]{xy}
\usepackage{tikz}
\usetikzlibrary{arrows}

\usepackage{multirow}

  \newcounter{mnote}
  \let\oldmarginpar\marginpar
    \renewcommand\marginpar[1]{\-\oldmarginpar[\raggedleft\footnotesize #1]%
    {\raggedright\footnotesize #1}}

\newtheorem{theorem}{Theorem}[section]
\newtheorem{lemma}[theorem]{Lemma}
\newtheorem{corollary}[theorem]{Corollary}

\newtheorem{remark}[theorem]{Remark}

\newcommand{\dd}{\,{\rm d}}
\newcommand{\bs}{\boldsymbol}

\newcommand{\curl}{\operatorname{curl}}
\renewcommand{\div}{\operatorname{div}}
\newcommand{\grad}{\operatorname{grad}}

\numberwithin{equation}{section}

\begin{document}
\title[Nonconforming Stokes complexes in three dimensions]{Nonconforming finite element Stokes complexes in three dimensions}
\author{Xuehai Huang}%
\address{School of Mathematics, Shanghai University of Finance and Economics, Shanghai 200433, China
}%
\email{huang.xuehai@sufe.edu.cn}%

\thanks{This work was supported by the National Natural Science Foundation of China Projects 12171300 and 11771338, the Natural Science Foundation of Shanghai 21ZR1480500, and the Fundamental Research Funds for the Central Universities 2019110066.}
\thanks{This paper has been accepted for publication in SCIENCE CHINA Mathematics.}


\begin{abstract}
Two nonconforming finite element Stokes complexes starting from the conforming Lagrange element and ending with the nonconforming $P_1$-$P_0$ element for the Stokes equation in three dimensions are studied. And commutative diagrams are also shown by combining nonconforming finite element Stokes complexes and interpolation operators. The lower order $\boldsymbol H(\textrm{grad}\textrm{curl})$-nonconforming finite element only has $14$ degrees of freedom, whose basis functions are explicitly given in terms of the barycentric coordinates. The $\boldsymbol H(\textrm{grad}\textrm{curl})$-nonconforming elements are applied to solve the quad-curl problem, and optimal convergence is derived. By the nonconforming finite element Stokes complexes, the mixed finite element methods of the quad-curl problem are decoupled into two mixed methods of the Maxwell equation and the nonconforming $P_1$-$P_0$ element method for the Stokes equation, based on which a fast solver is discussed. Numerical results are provided to verify the theoretical convergence rates.   
\end{abstract}


\maketitle


\section{Introduction}
In this paper we shall consider nonconforming finite element discretization of the following Stokes complex in three dimensions
\begin{equation}\label{eq:Stokescomplex3d}
\mathbb R\xrightarrow{\subset} H^1(\Omega)\xrightarrow{\nabla} \bs H(\grad\curl, \Omega) \xrightarrow{\curl} \boldsymbol H^1(\Omega; \mathbb{R}^3) \xrightarrow{\div} L^{2}(\Omega)\xrightarrow{}0,
\end{equation}
where $\bs H(\grad\curl, \Omega):=\{\bs v\in\bs H(\curl, \Omega): \curl\bs v\in \boldsymbol H^1(\Omega; \mathbb{R}^3)\}$.
Conforming finite element Stokes complexes on triangles and rectangles in two dimensions are devised in \cite{HuZhangZhang2020,ZhangWangZhang2019}.
And conforming finite element Stokes complexes on Alfeld split meshes in three dimensions are advanced in \cite{HuZhangZhang2020curcurl3d}.
We refer to \cite{BeiraodaVeigaDassiVacca2020} for a conforming virtual element discretization of the Stokes complex \eqref{eq:Stokescomplex3d}.
To the best of our knowledge, there is no finite element discretization of the Stokes complex \eqref{eq:Stokescomplex3d} in three dimensions using pure polynomials as shape functions in literature. Recently $\bs H(\grad\curl)$-conforming finite elements in three dimensions are constructed with $k\geq6$ in \cite{ZhangZhang2020}, whose space of shape functions includes all the polynomials with degree no more than $k$.
The number of the degrees of freedom for the lowest-order element in \cite{ZhangZhang2020} is $315$, which is reduced to $18$ by enriching the shape function space with macro-element bubble functions in \cite{HuZhangZhang2020curcurl3d}.
Nonconforming elements to discretize $\bs H(\grad\curl, \Omega)$ are another choices to reduce the large dimension of the conforming element spaces.
The $\bs H(\grad\curl)$-nonconforming Zheng-Hu-Xu element in \cite{ZhengHuXu2011} has only $20$ degrees of freedom, which is the first $\bs H(\grad\curl)$-nonconforming finite element.

We will construct an $\bs H(\grad\curl)$-nonconforming finite element possessing fewer degrees of freedom than the Zheng-Hu-Xu element, but preserving the same approximation error in energy norm. 
The finite element discretization of $\boldsymbol H^1(\Omega; \mathbb{R}^3) \times L^{2}(\Omega)$ in the Stokes complex \eqref{eq:Stokescomplex3d} should be a stable divergence-free pair for the Stokes equation, which suggests us to use the nonconforming linear element and piecewise constant to discretize $\boldsymbol H^1(\Omega; \mathbb{R}^3)$ and $L^{2}(\Omega)$ respectively.  
On the other hand, the direct sum decomposition $\mathbb P_{k}(K;\mathbb R^3)=\nabla\mathbb P_{k+1}(K) \oplus \big((\bs x-\bs x_K)\times\mathbb P_{k-1}(K;\mathbb R^3)\big)$~\cite{Arnold2018,ArnoldFalkWinther2006} implies the curl operator $\curl: (\bs x-\bs x_K)\times\mathbb P_{1}(K;\mathbb R^3) \to \mathbb P_{1}(K;\mathbb R^3)$ is injective. This motivates us to take the space of shape functions
$\bs W_k(K)=\nabla\mathbb P_{k+1}(K) \oplus \big((\bs x-\bs x_K)\times\mathbb P_{1}(K;\mathbb R^3)\big)$ with $k=0,1$.
Note that $\bs W_1(K)$ is exactly the space of shape functions of the Zheng-Hu-Xu element, hence we give a new understanding of the Zheng-Hu-Xu element by the space decomposition.
The dimension of $\bs W_0(K)$ is $14$, which is six fewer than the dimension of $\bs W_1(K)$.
The degrees of freedom $\mathcal N_0(K)$ for $\bs W_0(K)$ are given by
\begin{align*}
\int_e\bs v\cdot\bs t_e \dd s & \quad  \textrm{ on each }  e\in\mathcal E(K), \\
\int_F(\curl\bs v)\times\bs n\dd s & \quad \textrm{ on each }  F\in\mathcal F(K). 
\end{align*}
By comparing the degrees of freedom,
the lower order nonconforming element $(K, \bs W_0(K), \mathcal N_0(K))$ for $\bs H(\grad\curl, \Omega)$ is very similar as the Morley-Wang-Xu element \cite{WangXu2006} for $H^2(\Omega)$.
The explicit expressions of the basis functions of $\bs W_0(K)$ are shown in terms of the barycentric coordinates.

Then we combine the conforming $(k+1)$th order Lagrange element space $V_{h0}^g$, the $\bs H(\grad\curl)$-nonconforming finite element space $\boldsymbol W_{h0}$ including the Zheng-Hu-Xu element and the lower order one constructed in this paper, the nonconforming linear element space $\boldsymbol V_{h0}^s$, and the piecewise constant space $\mathcal Q_{h0}$ to build up the nonconforming finite element Stokes complexes
\begin{equation}\label{intro:Stokescomplex3dncfem0}
0\xrightarrow{\subset} V_{h0}^g\xrightarrow{\nabla} \bs W_{h0} \xrightarrow{\curl_h} \bs V_{h0}^s \xrightarrow{\div_h}  \mathcal Q_{h0}\xrightarrow{}0.
\end{equation}
The divergence-free subspace of the nonconforming linear element space $\boldsymbol V_{h0}^s$ is explicitly characterized due to this nonconforming finite element Stokes complex, which essentially extends the result of Falk and Morley \cite{FalkMorley1990} to three dimensions.
Recently this nonconforming finite element Stokes complex is applied to prove the quasi-orthogonality of the adaptive finite element method for the quad-curl problem in \cite{CaoChenHuang2020}. 
Furthermore, we develop the commutative diagram for Stokes complex \eqref{eq:Stokescomplex3d}
\begin{equation*}
\resizebox{0.915\hsize}{!}{$
\begin{array}{c}
\xymatrix{
 0 \ar[r]^-{\subset} & H_0^1(\Omega)\ar[d]^{I_h^{SZ}} \ar[r]^-{\nabla} & \boldsymbol H_0(\grad\curl, \Omega) \ar[d]^{\boldsymbol{\Pi}_h^{gc}} \ar[r]^-{\curl}
                & \boldsymbol{H}_0^1(\Omega; \mathbb{R}^3) \ar[d]^{\boldsymbol{I}_h^s}   \ar[r]^-{\div} & \ar[d]^{I_h^{L^2}}L_0^2(\Omega) \ar[r]^{} & 0 \\
 0 \ar[r]^-{\subset} & V_{h0}^g\ar[r]^-{\nabla} & \boldsymbol W_{h0} \ar[r]^{\curl_h}
                &  \boldsymbol V_{h0}^s   \ar[r]^{\div_h} &  \mathcal Q_{h0} \ar[r]^{}& 0    }
\end{array},
$}
\end{equation*}
where $I_h^{SZ}$ is the Scott-Zhang interpolation operator \cite{ScottZhang1990}, $\boldsymbol{\Pi}_h^{gc}$ is a quasi-interpolation operator, and both $\boldsymbol{I}_h^s$ and $I_h^{L^2}$ are the standard interpolation operators based on the degrees of freedom.

The $\bs H(\grad\curl)$-nonconforming element together with the Lagrange element is then applied to solve the quad-curl problem. The discrete Poincar\'e inequality is established for the $\bs H(\grad\curl)$-nonconforming element space $\boldsymbol W_{h0}$, as a result the coercivity on the weak divergence-free space follows. Then we acquire the discrete stability of the bilinear form from the evident discrete inf-sup condition, and derive the optimal convergence of the nonconforming mixed finite methods. 
Since the interpolation operator $\boldsymbol I_h^{gc}$ is not well-defined on $\bs H_0(\grad\curl, \Omega)$,
in the error analysis we exploit a quasi-interpolation operator $\boldsymbol{\Pi}_h^{gc}$ defined on $\bs H_0(\grad\curl, \Omega)$, which is constructed by combining a regular decomposition for the space $\bs H_0(\grad\curl, \Omega)$, the interpolation operator $\boldsymbol I_h^{gc}$ and the Scott-Zhang interpolation operator \cite{ScottZhang1990}.

By the nonconforming finite element Stokes complex \eqref{intro:Stokescomplex3dncfem0}, 
we equivalently decouple the mixed finite element methods of the quad-curl problem into two mixed methods of the Maxwell equation and the nonconforming $P_1$-$P_0$ element method for the Stokes equation, 
as the decoupling of the quad-curl problem  in the continuous level \cite{ChenHuang2018,Zhang2018a}. 
A fast solver based on this equivalent decoupling is discussed for the mixed finite element methods of the quad-curl problem.    

In addition to the Stokes complex \eqref{eq:Stokescomplex3d}, another kind of Stokes complex~\cite{JohnLinkeMerdonNeilanEtAl2017} is 
\begin{equation}\label{eq:anotherStokescomplex2d}
\mathbb R\xrightarrow{\subset} H^2(\Omega) \xrightarrow{\curl} \boldsymbol H^1(\Omega; \mathbb{R}^2) \xrightarrow{\div} L^{2}(\Omega)\xrightarrow{}0
\end{equation}
in two dimensions,
and
\begin{equation}\label{eq:anotherStokescomplex3d}
\mathbb R\xrightarrow{\subset} H^2(\Omega)\xrightarrow{\nabla} \bs H^1(\curl, \Omega) \xrightarrow{\curl} \boldsymbol H^1(\Omega; \mathbb{R}^3) \xrightarrow{\div} L^{2}(\Omega)\xrightarrow{}0
\end{equation}
in three dimensions,
where $\bs H^1(\curl, \Omega):=\{\bs v\in\boldsymbol H^1(\Omega; \mathbb{R}^3): \curl\bs v\in \boldsymbol H^1(\Omega; \mathbb{R}^3)\}$.
We refer to \cite{FalkMorley1990,MardalTaiWinther2002,AustinManteuffelMcCormick2004,Lee2010,FalkNeilan2013,GuzmanNeilan2012,GuzmanNeilan2014,Zhang2016,GuzmanLischkeNeilan2020} for some finite element discretizations of the Stokes complex \eqref{eq:anotherStokescomplex2d} in two dimensions,
and \cite{TaiWinther2006,GuzmanNeilan2012,Neilan2015} for some finite element discretizations of the Stokes complex \eqref{eq:anotherStokescomplex3d} in three dimensions.
While the finite elements corresponding to the Stokes complexes \eqref{eq:anotherStokescomplex2d}-\eqref{eq:anotherStokescomplex3d} are not suitable to discretize the quad-curl problem, since $\nabla H^1(\Omega)\subset\bs H(\grad\curl, \Omega)$ is not a subspace of $\bs H^1(\curl, \Omega)$.

The rest of this paper is organized as follows. In Section \ref{sec:gradcurlfem}, we devise a lower order $\bs H(\grad\curl)$-nonconforming finite element.
Nonconforming finite element Stokes complexes are developed in Section~\ref{sec:femstokescomplex}.
In Section~\ref{sec:mfem}, we propose the nonconforming mixed finite element methods for the quad-$\curl$ problem.
And the decoupling of the mixed finite element methods and a fast solver are discussed in Section \ref{sec:decoupling}.
Finally numerical results are presented in Section \ref{sec:numerresults}.

\section{The $\bs H(\grad\curl)$-nonconforming finite elements}\label{sec:gradcurlfem}

In this section we will present $\bs H(\grad\curl)$-nonconforming finite elements.


\subsection{Notation}

Given a bounded domain $G\subset\mathbb R^3$ and a nonnegative integer
$m$, let $H^m(G)$ be the usual Sobolev space of functions on $G$, and $\bs H^m(G;\mathbb R^3)$ the vector version of $H^m(G)$.
The corresponding norm and semi-norm are denoted, respectively, by $\|\cdot\|_{m,G}$ and $|\cdot|_{m,G}$. Let $(\cdot, \cdot)_G$ be the standard inner product on $L^2(G)$ or $\bs L^2(G;\mathbb R^3)$. If $G$ is $\Omega$, we abbreviate $\|\cdot\|_{m,G}$, $|\cdot|_{m,G}$ and $(\cdot, \cdot)_G$ by $\|\cdot\|_{m}$, $|\cdot|_{m}$ and $(\cdot, \cdot)$, respectively. Denote by $H_0^m(G) (\bs H_0^m(G; \mathbb R^3))$ the closure of $C_0^{\infty}(G) (\bs C_0^{\infty}(G; \mathbb R^3))$ with respect to the norm $\|\cdot\|_{m,G}$. 
Let $\mathbb P_m(G)$ stand for the set of all polynomials in $G$ with the total degree no more than $m$, and $\mathbb P_m(G;\mathbb R^3)$ be the vector version of $\mathbb P_m(G)$.
Let $Q_G^{m}: L^2(G)\to\mathbb P_{m}(G)$ be the $L^2$-orthogonal projector, and its vector version is denoted by $\bs Q_G^{m}$. Set $Q_G:=Q_G^0$.
The gradient operator,  curl operator and divergence  operator are denoted by $\nabla$, $\curl$ and $\div$ respectively. And define Sobolev spaces $\bs H(\curl, G)$, $\bs H_0(\curl, G)$, $\bs H(\div, G)$, $\bs H_0(\div, G)$ and $L_0^2(G)$ in the standard way.

Assume $\Omega\subset\mathbb R^3$ is a contractible polyhedron.
Let $\{\mathcal {T}_h\}_{h>0}$ be a regular family of tetrahedral meshes
of $\Omega$.  For each element $K\in\mathcal{T}_h$, denote by $\boldsymbol{n}_K $ the
unit outward normal vector to $\partial K$,  which will be abbreviated as $\boldsymbol{n}$ for simplicity.
Let $\mathcal{F}_h$, $\mathcal{F}^i_h$, $\mathcal{E}_h$ and $\mathcal{V}_h$ be the union of all faces, interior faces, edges and vertices
of the partition $\mathcal {T}_h$, respectively.
We fix a unit normal vector $\boldsymbol{n}_F$ for each face $F\in\mathcal{F}_h$, and a unit tangent vector $\boldsymbol{t}_e$ for each edge $e\in\mathcal{E}_h$.
For any $K\in\mathcal T_h$, denote by $\mathcal{F}(K)$, $\mathcal{E}(K)$ and $\mathcal{V}(K)$ the set of all faces, edges and vertices of $K$, respectively. 
For any $F\in\mathcal{F}_h$, let $\mathcal{E}(F)$ be the set of all edges of $F$. And for each $e\in\mathcal{E}(F)$, denote by $\bs n_{F,e}$ the unit vector
being parallel to $F$ and outward normal to $\partial F$. Set $\bs t_{F,e}:=\bs n_F\times\bs n_{F,e}$, where $\times$ is the exterior product.
For elementwise smooth function $\bs v$, define
\[
\|\bs v\|_{1,h}^2:=\sum_{K\in\mathcal T_h}\|\bs v\|_{1,K}^2,\quad |\bs v|_{1,h}^2:=\sum_{K\in\mathcal T_h}|\bs v|_{1,K}^2.
\]
Let $\nabla_h$, $\curl_h$ and $\div_h$ be the elementwise version of $\nabla$, $\curl$ and $\div$ with respect to $\mathcal T_h$.

\subsection{Nonconforming finite elements}
We focus on constructing nonconforming finite elements for the space $\bs H(\grad\curl, \Omega)$ in this subsection. 
To this end, recall the direct sum of the polynomial space \cite{Arnold2018,ArnoldFalkWinther2006}
\begin{equation}\label{eq:polyspacedecomposition}
\mathbb P_{k}(K;\mathbb R^3)=\nabla\mathbb P_{k+1}(K) \oplus \big((\bs x-\bs x_K)\times\mathbb P_{k-1}(K;\mathbb R^3)\big) \quad \forall~K\in\mathcal T_h,
\end{equation}
where $\bs x_K$ is the barycenter of $K$.
The decomposition \eqref{eq:polyspacedecomposition} implies that $\curl: (\bs x-\bs x_K)\times\mathbb P_{k-1}(K;\mathbb R^3)\to\mathbb P_{k-1}(K;\mathbb R^3)$ is injective. 
We intend to use the nonconforming linear element to discretize $\boldsymbol H^1(\Omega; \mathbb{R}^3)$, then the  decomposition \eqref{eq:polyspacedecomposition} and the complex \eqref{eq:Stokescomplex3d} motivate us that the space of shape functions to discrete $\bs H(\grad\curl, \Omega)$ should include $(\bs x-\bs x_K)\times\mathbb P_{1}(K;\mathbb R^3)$. The direct sum in \eqref{eq:polyspacedecomposition} also suggests to enrich $(\bs x-\bs x_K)\times\mathbb P_{1}(K;\mathbb R^3)$ with $\nabla\mathbb P_{l}(K)$ for some positive integer $l$ to get the space of shape functions.
Hence
for each $K\in\mathcal T_h$, define the space of shape functions as
$$
\bs W_k(K):=\nabla\mathbb P_{k+1}(K) \oplus \big((\bs x-\bs x_K)\times\mathbb P_{1}(K;\mathbb R^3)\big) \quad\textrm{ for } k=0,1.
$$
By the decomposition \eqref{eq:polyspacedecomposition}, we have $\mathbb P_k(K;\mathbb R^3)\subset \bs W_k(K)\subset \mathbb P_2(K;\mathbb R^3)$,
and
\[
\dim\bs W_k(K)=\begin{cases}
14, & k=0,\\
20, & k=1.
\end{cases}
\]
Then choose the following local degrees of freedom $\mathcal N_k(K)$
\begin{align}
\int_e\bs v\cdot\bs t_e q\dd s & \quad \forall~q\in \mathbb P_k(e) \textrm{ on each }  e\in\mathcal E(K), \label{gradcurlncfmdof1}\\
\int_F(\curl\bs v)\times\bs n\dd s & \quad \textrm{ on each }  F\in\mathcal F(K). \label{gradcurlncfmdof2}
\end{align}
The degrees of freedom \eqref{gradcurlncfmdof1}-\eqref{gradcurlncfmdof2} are inspired by the degrees of freedom of nonconforming linear element and the N\'ed\'elec element \cite{Nedelec1980,Nedelec1986}.
Note that the triple $(K, \bs W_1(K), \mathcal N_1(K))$ is exactly the nonconforming finite element in \cite{ZhengHuXu2011}.
In this paper we will embed this nonconforming finite element into the discrete Stokes complex. And we also construct the lowest order triple $(K, \bs W_0(K), \mathcal N_0(K))$.

\begin{lemma}\label{lem:Whunisol}
The degrees of freedom \eqref{gradcurlncfmdof1}-\eqref{gradcurlncfmdof2} are unisolvent for the shape function space $\bs W_k(K)$.
\end{lemma}
\begin{proof}
Notice that the number of the degrees of freedom \eqref{gradcurlncfmdof1}-\eqref{gradcurlncfmdof2} is same as the dimension of $\bs W_k(K)$.  It is sufficient to show that $\bs v=\bs0$ for any $\bs v\in\bs W_k(K)$ with vanishing degrees of freedom \eqref{gradcurlncfmdof1}-\eqref{gradcurlncfmdof2}.

For each $F\in\mathcal F(K)$, apply the integration by parts on face $F$ to obtain
\begin{align}
\int_F(\curl\bs v)\cdot\bs n_{F}\dd s&=\int_F\div(\bs v\times \bs n_F)\dd s=\sum_{e\in\mathcal E(F)}\int_e(\bs v\times \bs n_F)|_F\cdot\bs n_{F,e}\dd s \notag\\
&=\sum_{e\in\mathcal E(F)}\int_e\bs v\cdot( \bs n_F\times\bs n_{F,e})\dd s=\sum_{e\in\mathcal E(F)}\int_e\bs v\cdot\bs t_{F,e}\dd s. \label{eq:20220207}
\end{align}
We get from the vanishing degrees of freedom \eqref{gradcurlncfmdof1} that $\int_F(\curl\bs v)\cdot\bs n_{F}\dd s=0$,
which together with the vanishing degrees of freedom \eqref{gradcurlncfmdof2} implies
\[
\int_F\curl\bs v\dd s=\bs0.
\]
Since $\curl\bs v\subseteq \mathbb P_{1}(K;\mathbb R^3)$, we acquire from the unisolvence of the nonconforming linear element that $\curl\bs v=\bs 0$.
Employing the fact that $\curl: (\bs x-\bs x_K)\times\mathbb P_{1}(K;\mathbb R^3)\to\mathbb P_{1}(K;\mathbb R^3)$ is injective, there exists $q\in \mathbb P_{k+1}(K)$ such that $\bs v=\nabla q$.
By the vanishing degrees of freedom \eqref{gradcurlncfmdof1}, it holds $\partial_{t_e}q=0$, which implies that we can choose $q\in \mathbb P_{k+1}(K)$ such that $q|_{e}=0$ for each $e\in \mathcal E(K)$.
Noting that $k=0,1$, we acquire $q=0$ and $\bs v=\bs0$.
\end{proof}

By comparing the degrees of freedom,
the lower order nonconforming element $(K, \bs W_0(K), \mathcal N_0(K))$ for $\bs H(\grad\curl, \Omega)$ is very similar as the Morley-Wang-Xu element \cite{WangXu2006} for $H^2(\Omega)$.

Next we give a norm equivalence of space $\bs W_k(K)$. To this end,
recall the Poincar\'e operator $\mathcal K_K: \mathbb P_{1}(K;\mathbb R^3)\to (\bs x-\bs x_K)\times\mathbb P_{1}(K;\mathbb R^3)$ in \cite{GopalakrishnanDemkowicz2004,Hiptmair1999}
$$
\mathcal K_K\boldsymbol q:=-(\boldsymbol{x}-\boldsymbol{x}_K)\times\int_0^1t\,\boldsymbol q(t(\boldsymbol{x}-\boldsymbol{x}_K)+\boldsymbol{x}_K)\dd t.
$$
It holds the identity \cite[Theorem~2.1]{GopalakrishnanDemkowicz2004}
\begin{equation}\label{eq:poincareoperatorprop1}
\curl\mathcal K_K(\curl\boldsymbol v)=\curl\boldsymbol v\quad\forall~\boldsymbol v\in \bs W_k(K).
\end{equation}
By the inverse inequality, we have
\begin{equation}\label{eq:poincareoperatorprop2}
\|\mathcal K_K\boldsymbol q\|_{0,K}\lesssim h_K^{5/2}\|\boldsymbol q\|_{L^{\infty}(K)}\lesssim h_K\|\boldsymbol q\|_{0,K}\quad\forall~\boldsymbol q\in\mathbb P_{1}(K;\mathbb R^3).
\end{equation}

\begin{lemma}\label{lem:WkKdecomp}
For $\boldsymbol v\in\bs W_k(K)$, there exists $q\in\mathbb P_{k+1}(K)$ such that
\begin{equation}\label{eq:WkKdecomp}
\boldsymbol v= \nabla q + \mathcal K_K(\curl\boldsymbol v),
\end{equation}
\begin{equation}\label{eq:WkKdecompprop}
\|q\|_{0,K}^2\lesssim h_K^4\|\curl\boldsymbol v\|_{0,K}^2+h_K^4\sum_{e\in\mathcal E(K)}\|Q_e^k(\boldsymbol v\cdot\boldsymbol t_e)\|_{0,e}^2.
\end{equation}
\end{lemma}
\begin{proof}
Take a vertex $\delta\in\mathcal V(K)$. Due to \eqref{eq:poincareoperatorprop1}, $\bs v-\mathcal K_K(\curl\boldsymbol v)\in\bs W_k(K)\cap\ker(\curl)$, which means $\bs v-\mathcal K_K(\curl\boldsymbol v)\in\nabla\mathbb P_{k+1}(K)$. Choose $q\in \mathbb P_{k+1}(K)$ such that $\bs v-\mathcal K_K(\curl\boldsymbol v)=\nabla q$ and $q(\delta)=0$. By the fact $q\in \mathbb P_{2}(K)$, the norm equivalence of Lagrange element and the inverse inequality,
\begin{align*}
\|q\|_{0,K}^2&\lesssim h_K^2\sum_{e\in\mathcal E(K)}\|q\|_{0,e}^2\lesssim h_K^3\sum_{e\in\mathcal E(K)}\|q\|_{L^{\infty}(e)}^2=h_K^3\sum_{e\in\mathcal E(K)}\|q(\bs x)-q(\delta)\|_{L^{\infty}(e)}^2 \\
&\lesssim h_K^3\sum_{e\in\mathcal E(K)}h_e^2\|\partial_tq\|_{L^{\infty}(e)}^2\lesssim h_K^4\sum_{e\in\mathcal E(K)}\|\partial_tq\|_{0,e}^2.
\end{align*}
Since $\partial_tq=Q_e^k(\partial_tq)=Q_e^k(\boldsymbol v\cdot\boldsymbol t_e) + Q_e^k(\mathcal K_K\big(\curl\boldsymbol v)\cdot \boldsymbol t_e\big)$ on edge $e$, we get from the inverse inequality that
\begin{align*}
\|q\|_{0,K}^2&\lesssim h_K^4\sum_{e\in\mathcal E(K)}(\|Q_e^k(\boldsymbol v\cdot\boldsymbol t_e)\|_{0,e}^2+\|\mathcal K_K(\curl\boldsymbol v)\|_{0,e}^2) \\
&\lesssim h_K^2\|\mathcal K_K(\curl\boldsymbol v)\|_{0,K}^2+h_K^4\sum_{e\in\mathcal E(K)}\|Q_e^k(\boldsymbol v\cdot\boldsymbol t_e)\|_{0,e}^2.
\end{align*}
Finally we conclude \eqref{eq:WkKdecompprop} from \eqref{eq:poincareoperatorprop2}.
\end{proof}

\begin{lemma}
For $\boldsymbol v\in\bs W_k(K)$,
it holds the norm equivalence
\begin{equation}\label{eq:WkKnormequivalence}
\|\boldsymbol v\|_{0,K}^2\eqsim h_K^2\sum_{e\in\mathcal E(K)}\|Q_e^k(\boldsymbol v\cdot\boldsymbol t_e)\|_{0,e}^2 + h_K^3\sum_{F\in\mathcal F(K)}\|\boldsymbol{Q}_F^0((\curl\bs v)\times\bs n)\|_{0,F}^2.
\end{equation}
\end{lemma}
\begin{proof}
Since $\curl\boldsymbol v\in\mathbb P_{1}(K;\mathbb R^3)$, by the norm equivalence of the nonconforming $P_1$ element,
\begin{align*}  
\|\curl\boldsymbol v\|_{0,K}^2&\lesssim h_K\sum_{F\in\mathcal F(K)}\|\boldsymbol{Q}_F^0(\curl\bs v)\|_{0,F}^2 \\
&\lesssim h_K\sum_{F\in\mathcal F(K)}\left(\|\boldsymbol{Q}_F^0((\curl\bs v)\times\bs n)\|_{0,F}^2+\|Q_F^0((\curl\bs v)\cdot\bs n)\|_{0,F}^2\right).
\end{align*}
From \eqref{eq:20220207} we get
\begin{align*}
h_K\|Q_F^0((\curl\bs v)\cdot\bs n)\|_{0,F}^2&\lesssim h_K^3|Q_F^0((\curl\bs v)\cdot\bs n)|^2\lesssim h_K\sum_{e\in\mathcal E(F)}|Q_e^0(\boldsymbol v\cdot\boldsymbol t_e)|^2 \\
&\lesssim \sum_{e\in\mathcal E(F)}\|Q_e^0(\boldsymbol v\cdot\boldsymbol t_e)\|_{0,e}^2\leq \sum_{e\in\mathcal E(F)}\|Q_e^k(\boldsymbol v\cdot\boldsymbol t_e)\|_{0,e}^2.
\end{align*}
Combining the last two inequalities yields
\begin{equation}\label{eq:202202071}
\|\curl\boldsymbol v\|_{0,K}^2\lesssim \sum_{e\in\mathcal E(K)}\|Q_e^k(\boldsymbol v\cdot\boldsymbol t_e)\|_{0,e}^2 + h_K\sum_{F\in\mathcal F(K)}\|\boldsymbol{Q}_F^0((\curl\bs v)\times\bs n)\|_{0,F}^2.
\end{equation}
Applying Lemma~\ref{lem:WkKdecomp} to $\boldsymbol{v}$, we derive from \eqref{eq:WkKdecomp}, the inverse inequality, \eqref{eq:poincareoperatorprop2} and~\eqref{eq:WkKdecompprop} that
\begin{align*}
\|\boldsymbol v\|_{0,K}^2&\leq 2\|\nabla q\|_{0,K}^2 + 2\|\mathcal K_K(\curl\boldsymbol v)\|_{0,K}^2\lesssim h_K^{-2}\|q\|_{0,K}^2 + h_K^{2}\|\curl\boldsymbol v\|_{0,K}^2 \\
&\lesssim h_K^2\|\curl\boldsymbol v\|_{0,K}^2+h_K^2\sum_{e\in\mathcal E(K)}\|Q_e^k(\boldsymbol v\cdot\boldsymbol t_e)\|_{0,e}^2.
\end{align*}
Then we acquire from \eqref{eq:202202071} that
$$
\|\boldsymbol v\|_{0,K}^2\lesssim h_K^2\sum_{e\in\mathcal E(K)}\|Q_e^k(\boldsymbol v\cdot\boldsymbol t_e)\|_{0,e}^2 + h_K^3\sum_{F\in\mathcal F(K)}\|\boldsymbol{Q}_F^0((\curl\bs v)\times\bs n)\|_{0,F}^2.
$$
The another side of \eqref{eq:WkKnormequivalence} follows from the inverse inequality.
\end{proof}

\subsection{Basis functions}
We will figure out the basis functions of $\bs W_0(K)$ in this subsection. We refer to \cite{ZhengHuXu2011} for the basis functions of $\bs W_1(K)$.
Let $\lambda_1$, $\lambda_2$, $\lambda_3$ and $\lambda_4$ be the barycentric coordinates of point $\bs x$ with respect to the vertices $\bs x_{1}, \bs x_{2}, \bs x_{3}$ and $\bs x_{4}$ of the tetrahedron $K$ respectively.
Let $F_l$ be the face of $K$ opposite to $\bs x_l$. And the vertices of $F_l$ denoted by $\bs x_{l_1}$, $\bs x_{l_2}$ and $\bs x_{l_3}$ with $l_1<l_2<l_3$.
Set $\bs t_{ij}:=\bs x_{j}-\bs x_{i}$, which is a tangential vector to the edge $e_{ij}$ with vertices $\bs x_{i}$ and $\bs x_{j}$, and similarly define other tangential vectors with different subscripts.
For ease of presentation,
let
\[
M_{e_{ij}}(\bs v):=\frac{1}{|e_{ij}|}\int_{e_{ij}}\bs v\cdot\bs t_{ij} \dd s,\quad \bs M_{F_l}(\bs v):=\int_{F_l}(\curl\bs v)\times\bs n_l\dd s,
\]
\[
M_{F_l,1}(\bs v):=\frac{1}{2|F_l|(\nabla\lambda_{l_1}\times\nabla\lambda_{l_2})\cdot\bs n_{l}}\int_{F_l}(\curl\bs v)\cdot((\bs n_{l}\times\nabla\lambda_{l_2})\times\bs n_{l})\dd s,
\]
\[
M_{F_l,2}(\bs v):=\frac{1}{2|F_l|(\nabla\lambda_{l_2}\times\nabla\lambda_{l_1})\cdot\bs n_{l}}\int_{F_l}(\curl\bs v)\cdot((\bs n_{l}\times\nabla\lambda_{l_1})\times\bs n_{l})\dd s.
\]
The degrees of freedom $M_{F_l,1}(\bs v)$ and $M_{F_l,2}(\bs v)$ are equivalent to $\bs M_{F_l}(\bs v)$, i.e. \eqref{gradcurlncfmdof2}.
\subsubsection{Basis functions corresponding to the face degrees of freedom}
Define 
\begin{align*}
\bs\varphi_{F_l,i}&:=\frac{1}{4}(8\lambda_l-3)(\bs x-\bs x_K)\times(\bs n_{l}\times\nabla\lambda_{l_i}) +\frac{1}{4}(\bs x_l-\bs x_K)\cdot\bs n_{l}\nabla\lambda_{l_i} + \frac{1}{16}\bs n_{l} \\
&\;=\frac{1}{16}(8\lambda_l-3)[(4\lambda_{l_i}-1)\bs n_{l}-4(\bs x-\bs x_K)\cdot\bs n_{l}\nabla\lambda_{l_i}] \\
&\quad\;\;+\frac{1}{4}(\bs x_l-\bs x_K)\cdot\bs n_{l}\nabla\lambda_{l_i} +\frac{1}{16}\bs n_{l} 
\end{align*}
for $i=1,2$. We will show that $\bs\varphi_{F_l,1}$ and $\bs\varphi_{F_l,2}$ are the basis functions being dual to $M_{F_l,1}(\bs v)$ and $M_{F_l,2}(\bs v)$.


\begin{lemma}
Functions $\bs\varphi_{F_l,1}$ and $\bs\varphi_{F_l,2}$ are the basis functions of $\bs W_0(K)$ being dual to $M_{F_l,1}(\bs v)$ and $M_{F_l,2}(\bs v)$, respecrtively. That is
\begin{equation}\label{eq:facebasisdof1}
M_{e}(\bs\varphi_{F_l,1})=M_{e}(\bs\varphi_{F_l,2})=0\quad\forall~e\in\mathcal E(K),
\end{equation}
\begin{equation}\label{eq:facebasisdof2}
\bs M_{F}(\bs\varphi_{F_l,1})=\bs M_{F}(\bs\varphi_{F_l,2})=\bs 0\quad\forall~F\in\mathcal F(K)\backslash \{F_l\},
\end{equation}
\begin{equation}\label{eq:facebasisdof3}
M_{F_l,2}(\bs\varphi_{F_l,1})=M_{F_l,1}(\bs\varphi_{F_l,2})=0,\quad M_{F_l,1}(\bs\varphi_{F_l,1})=M_{F_l,2}(\bs\varphi_{F_l,2})=1.
\end{equation}
\end{lemma}
\begin{proof}
Apparently $\bs\varphi_{F_l,1}\cdot \bs t_{l_2l_3}=0$.
By $\bs n_{l}\cdot\bs t_{l_1l_2}=0$, $\nabla\lambda_{l_1}\cdot\bs t_{l_1l_2}=-1$ and $\lambda_l|_{e_{l_1l_2}}=0$, we get
\begin{align*}
M_{e_{l_1l_2}}(\bs\varphi_{F_l,1})&=\frac{1}{4|e_{l_1l_2}|}\int_{e_{l_1l_2}}\big((8\lambda_l-3)(\bs x-\bs x_K)\cdot\bs n_l - (\bs x_l-\bs x_K)\cdot\bs n_l\big)\dd s \\
&=-\frac{3}{4}(\bs x_{l_1}-\bs x_K)\cdot\bs n_l - \frac{1}{4}(\bs x_l-\bs x_K)\cdot\bs n_l.
\end{align*}
Noting that
$
\bs x_l-\bs x_K+3(\bs x_{l_1}-\bs x_K)=\bs x_l+3\bs x_{l_1}-4\bs x_K=2\bs x_{l_1}-\bs x_{l_2}-\bs x_{l_3}
$
is parallel to face $F_l$, we have $\big(\bs x_l-\bs x_K+3(\bs x_{l_1}-\bs x_K)\big)\cdot\bs n_l=0$. Hence
$$
M_{e_{l_1l_2}}(\bs\varphi_{F_l,1})=0.
$$
Since $\bs n_l=\frac{\bs n_l\cdot\nabla\lambda_l}{|\nabla\lambda_l|^2}\nabla\lambda_l$, $4(\bs x-\bs x_K)\cdot\nabla\lambda_l=4\lambda_l-1$ and $(\lambda_{l_1}+\lambda_l)|_{e_{l_1l}}=1$, it follows
\begin{align*}
M_{e_{l_1l}}(\bs\varphi_{F_l,1})&=\frac{\bs n_l\cdot\nabla\lambda_l}{8|e_{l_1l}||\nabla\lambda_l|^2}\int_{e_{l_1l}}(8\lambda_l-3) \dd s + \frac{\bs n_l\cdot\nabla\lambda_l}{16|\nabla\lambda_l|^2}\big(1-4(\bs x_l-\bs x_K)\cdot\nabla\lambda_l\big) \\
&=\frac{\bs n_l\cdot\nabla\lambda_l}{16|\nabla\lambda_l|^2}\big(3-4(\bs x_l-\bs x_K)\cdot\nabla\lambda_l\big) =0.
\end{align*}
Similarly we can show that $M_{e}(\bs\varphi_{F_l,1})=0$ for other edges and $M_{e}(\bs\varphi_{F_l,2})=0$. Hence~\eqref{eq:facebasisdof1} holds.

On the other hand, by the identity $\curl((\bs x-\bs x_K)\times \bs q)=(\bs x-\bs x_K)\div\bs q - ((\bs x-\bs x_K)\cdot\nabla) \bs q-2\bs q$,
we have for $i=1,2$,
\[
\curl\bs\varphi_{F_l,i}=\frac{1}{4}\curl\big((8\lambda_l-3)(\bs x-\bs x_K)\times(\bs n_{l}\times\nabla\lambda_{l_i})\big)=2(1-3\lambda_l)\bs n_l\times\nabla\lambda_{l_i}.
\]
We conclude \eqref{eq:facebasisdof2}-\eqref{eq:facebasisdof3}
by the fact that $1-3\lambda_l$ is the basis function of the nonconforming $P_1$ element.
\end{proof}

\subsubsection{Basis functions corresponding to the edge degrees of freedom}
Next we construct the basis function corresponding to the degree of freedom $M_{e_{ij}}(\bs v)$. 
Recall the basis function of the lowest order N\'ed\'elec element of the first kind
$\lambda_i\nabla \lambda_j-\lambda_j\nabla \lambda_i$. Thanks to \eqref{eq:facebasisdof1}-\eqref{eq:facebasisdof2},  function $\lambda_i\nabla \lambda_j-\lambda_j\nabla \lambda_i$ can be modified by $\bs\varphi_{F_l,1}$ and $\bs\varphi_{F_l,2}$ to derive the basis function of $\bs W_0(K)$ corresponding to the degree of freedom $M_{e_{ij}}(\bs v)$.  
\begin{lemma}
Let
\[
\bs\varphi_{e_{ij}}:=\lambda_i\nabla \lambda_j-\lambda_j\nabla \lambda_i+\sum_{l=1}^4(c_{l,1}^{ij}\bs\varphi_{F_l,1} + c_{l,2}^{ij}\bs\varphi_{F_l,2})
\]
with constants
$$c_{l,1}^{ij}:=\frac{1}{(\nabla\lambda_{l_1}\times\nabla\lambda_{l_2})\cdot\bs n_{l}}(\nabla \lambda_j\times \nabla \lambda_i)\cdot((\bs n_{l}\times\nabla\lambda_{l_2})\times\bs n_{l}),$$ 
$$c_{l,2}^{ij}:=\frac{1}{(\nabla\lambda_{l_2}\times\nabla\lambda_{l_1})\cdot\bs n_{l}}(\nabla \lambda_j\times \nabla \lambda_i)\cdot((\bs n_{l}\times\nabla\lambda_{l_1})\times\bs n_{l}).$$ Then
\[
M_{e_{ij}}(\bs\varphi_{e_{ij}})=1,\quad M_{e}(\bs\varphi_{e_{ij}})=0, \quad
\bs M_{F}(\bs\varphi_{e_{ij}})=\bs0.
\]
for each $e\in\mathcal E(K)\backslash\{e_{ij}\}$ and $F\in\mathcal F(K)$.
\end{lemma}
\begin{proof}
The identities $M_{e_{ij}}(\bs\varphi_{e_{ij}})=1$ and $M_{e}(\bs\varphi_{e_{ij}})=0$ follow from \eqref{eq:facebasisdof1} and the fact
\[
M_{e_{ij}}(\lambda_i\nabla \lambda_j-\lambda_j\nabla \lambda_i)=1,\quad M_{e}(\lambda_i\nabla \lambda_j-\lambda_j\nabla \lambda_i)=0\quad\forall~e\in\mathcal E(K)\backslash\{e_{ij}\}.
\]
On the other hand, we get from \eqref{eq:facebasisdof2}-\eqref{eq:facebasisdof3} and
$
\curl(\lambda_i\nabla \lambda_j-\lambda_j\nabla \lambda_i)=2\nabla \lambda_i\times\nabla \lambda_j
$ that
\[
M_{F_l,r}(\bs\varphi_{e_{ij}})=M_{F_l,r}(\lambda_i\nabla \lambda_j-\lambda_j\nabla \lambda_i) + c_{l,r}^{ij}=0
\]
for $r=1, 2$.
\end{proof}

In summary, we arrive at the basis functions being dual to the degrees of freedom $M_{F_{l},1}(\bs v)$, $M_{F_{l},2}(\bs v)$ and $M_{e_{ij}}(\bs v)$:
\begin{enumerate}
\item Two basis functions on each face $F_l$ ($1\leq l\leq 4$)
\[
\bs\varphi_{F_l,i}=\frac{1}{4}(8\lambda_l-3)(\bs x-\bs x_K)\times(\bs n_{l}\times\nabla\lambda_{l_i}) +\frac{1}{4}(\bs x_l-\bs x_K)\cdot\bs n_{l}\nabla\lambda_{l_i} + \frac{1}{16}\bs n_{l}
\]
for $i=1,2$, where $\bs x_K$ is the barycenter of $K$.
\item One basis function on each edge $e_{ij}$ ($1\leq i<j\leq 4$)
\[
\bs\varphi_{e_{ij}}=\lambda_i\nabla \lambda_j-\lambda_j\nabla \lambda_i+\sum_{l=1}^4(c_{l,1}^{ij}\bs\varphi_{F_l,1} + c_{l,2}^{ij}\bs\varphi_{F_l,2})
\]
with constants
$$c_{l,1}^{ij}:=\frac{1}{(\nabla\lambda_{l_1}\times\nabla\lambda_{l_2})\cdot\bs n_{l}}(\nabla \lambda_j\times \nabla \lambda_i)\cdot((\bs n_{l}\times\nabla\lambda_{l_2})\times\bs n_{l}),$$ 
$$c_{l,2}^{ij}:=\frac{1}{(\nabla\lambda_{l_2}\times\nabla\lambda_{l_1})\cdot\bs n_{l}}(\nabla \lambda_j\times \nabla \lambda_i)\cdot((\bs n_{l}\times\nabla\lambda_{l_1})\times\bs n_{l}).$$ 
\end{enumerate}

\section{Nonconforming finite element Stokes complexes}\label{sec:femstokescomplex}
We will consider the nonconforming finite element discretization of the Stokes complex \eqref{eq:Stokescomplex3d} in this section.
The homogeneous version of the Stokes complex \eqref{eq:Stokescomplex3d} is
\begin{equation*}
0\xrightarrow{\subset} H_0^1(\Omega)\xrightarrow{\nabla} \bs H_0(\grad\curl, \Omega) \xrightarrow{\curl} \boldsymbol H_0^1(\Omega; \mathbb{R}^3) \xrightarrow{\div} L_0^{2}(\Omega)\xrightarrow{}0,
\end{equation*}
where $\bs H_0(\grad\curl, \Omega):=\{\bs v\in\bs H_0(\curl, \Omega): \curl\bs v\in \boldsymbol H_0^1(\Omega; \mathbb{R}^3)\}$.
Since $\bs H_0(\curl, \Omega)\cap \bs H_0(\div, \Omega)=\bs H_0^1( \Omega; \mathbb R^3)$, it holds $\bs H_0(\grad\curl, \Omega)=\bs H_0(\curl^2, \Omega)$, where
\[
\bs H_0(\curl^2, \Omega):=\{\bs v\in\bs H_0(\curl, \Omega): \curl\bs v\in \bs H_0(\curl, \Omega)\}.
\]

We can use the Lagrange element, the nonconforming linear element and piecewise constant to discretize $H^1(\Omega)$, $\boldsymbol H^1(\Omega; \mathbb{R}^3)$ and $L^{2}(\Omega)$ in the Stokes complex~\eqref{eq:Stokescomplex3d}, respectively.
Take the Lagrange element space 
\[
V_h^g:=\{v_h\in H^{1}(\Omega): v_h|_K\in\mathbb P_{k+1}(K) \textrm{ for each } K\in\mathcal T_h\}
\]
with $k=0,1$, the nonconforming linear element space 
\begin{align*}
\bs V_h^s:=\Big\{\bs v_h\in \boldsymbol L^2(\Omega; \mathbb{R}^3): &\,\bs v_h|_K\in\mathbb P_{1}(K;\mathbb R^3) \textrm{ for each } K\in\mathcal T_h, \\
&\textrm{ and } \int_F\llbracket \bs v_h\rrbracket\dd s=\bs 0  \textrm{ for each } F\in\mathcal F_h^i\Big\},
\end{align*}
and the piecewise constant space
\[
\mathcal Q_h:=\{q_h\in L^2(\Omega): q_h|_K\in\mathbb P_0(K) \textrm{ for each } K\in\mathcal T_h\}.
\]
Here $\llbracket \bs v_h\rrbracket$ is the jump of $\bs v_h$ across $F$.
Define the global $\bs H(\grad\curl)$-nonconforming element space
\begin{align*}
\bs W_h:=\{\bs v_h\in \boldsymbol L^2(\Omega; \mathbb{R}^3):&\, \bs v_h|_K\in\bs W_k(K) \textrm{ for each } K\in\mathcal T_h,  \textrm{ and all the }\\
&\textrm{ degrees of freedom \eqref{gradcurlncfmdof1}-\eqref{gradcurlncfmdof2} are single-valued}\}.
\end{align*}
According to the proof of Lemma~\ref{lem:Whunisol}, it holds
\begin{equation}\label{eq:curlWhweakcontinuity}
\int_F\llbracket\curl\bs v_h\rrbracket\dd s=\bs0\quad\forall~\bs v_h\in\bs W_h, \; F\in\mathcal F_h^i.
\end{equation}

To prove the exactness of the nonconforming discrete Stokes complexes, we need the help of the N\'ed\'elec element spaces \cite{Nedelec1980,Nedelec1986}
\[
\bs V_h^c:=\{\bs v_h\in \bs H(\curl, \Omega): \bs v_h|_K\in\bs V_k^c(K) \textrm{ for each } K\in\mathcal T_h\},
\]
where $\bs V_k^c(K):=\mathbb P_{k}(K;\mathbb R^3)+ (\bs x-\bs x_K)\times\mathbb P_{0}(K;\mathbb R^3)$ with $k=0,1$.
Apparently $\bs V_k^c(K)\subset\bs W_k(K)$.
The degrees of freedom for $\bs V_k^c(K)$ are
\begin{align}
\int_e\bs v\cdot\bs t_e q\dd s & \quad \forall~q\in \mathbb P_k(e) \textrm{ on each }  e\in\mathcal E(K). \label{Nedelecdof1}
\end{align}
It is observed that the degrees of freedom \eqref{Nedelecdof1} are exactly same as \eqref{gradcurlncfmdof1}.
By the finite element de Rham complexes \cite{Arnold2018,ArnoldFalkWinther2006}, we have
\begin{equation}\label{eq:deRhamexact1}
\bs V_h^c\cap\ker(\curl)=\nabla V_h^g.
\end{equation}
The notation $\ker(\mathcal A)$ means the kernel space of the operator $\mathcal A$.

\begin{lemma}\label{lem:k01exact1}
It holds
\[
\bs W_h\cap\ker(\curl_h)=\nabla V_h^g.
\]
\end{lemma}
\begin{proof}
Since $\curl: (\bs x-\bs x_K)\times\mathbb P_{1}(K;\mathbb R^3) \to \mathbb P_{1}(K;\mathbb R^3)$ is injective \cite{Arnold2018,ArnoldFalkWinther2006}, we have
\[
\bs W_h\cap\ker(\curl_h)=\{\bs v_h\in \bs W_h: \bs v_h|_K\in\nabla\mathbb P_{k+1}(K) \textrm{ for each } K\in\mathcal T_h\},
\]
\[
\bs V_h^c\cap\ker(\curl_h)=\{\bs v_h\in \bs V_h^c: \bs v_h|_K\in\nabla\mathbb P_{k+1}(K) \textrm{ for each } K\in\mathcal T_h\}.
\]
Noting that the degrees of freedom \eqref{gradcurlncfmdof1} and \eqref{Nedelecdof1} are the same, it follows $\bs W_h\cap\ker(\curl_h)=\bs V_h^c\cap\ker(\curl_h)$.
Thus we finish the proof from \eqref{eq:deRhamexact1}.
\end{proof}

\begin{lemma}
The nonconforming discrete Stokes complex
\begin{equation}\label{eq:Stokescomplex3dncfem}
\mathbb R\xrightarrow{\subset} V_h^g\xrightarrow{\nabla} \bs W_h \xrightarrow{\curl_h} \bs V_h^s \xrightarrow{\div_h}  \mathcal Q_h\xrightarrow{}0
\end{equation}
is exact.
\end{lemma}
\begin{proof}
We refer to  \cite{CrouzeixRaviart1973,BoffiBrezziFortin2013}  for $\div_h\bs V_h^s=\mathcal Q_h$ and Lemma~\ref{lem:k01exact1} for $\bs W_h\cap\ker(\curl_h)=\nabla V_h^g$.
By the definition of $\bs W_h$, apparently we have from \eqref{eq:curlWhweakcontinuity} that
\[
\curl_h\bs W_h\subseteq \bs V_h^s \cap\ker(\div_h).
\]
Then we prove
\[
\curl_h\bs W_h=\bs V_h^s \cap\ker(\div_h)
\]
by counting the dimensions of these spaces. Indeed, we have
\begin{align*}
\dim\curl_h\bs W_h&=\dim\bs W_h-\dim V_h^g+1 \\
&=(k+1)\#\mathcal E_h+2\#\mathcal F_h -\#\mathcal V_h-k\#\mathcal E_h+1 \\
&=\#\mathcal E_h+2\#\mathcal F_h-\#\mathcal V_h+1,
\end{align*}
\[
\dim\bs V_h^s \cap\ker(\div_h)=\dim\bs V_h^s-\dim\mathcal Q_h=3\#\mathcal F_h-\#\mathcal T_h.
\]
Finally apply the Euler's formula $\#\mathcal V_h-\#\mathcal E_h+\#\mathcal F_h-\#\mathcal T_h=1$ to end the proof.
\end{proof}

\begin{corollary}
The nonconforming discrete Stokes complex with homogeneous boundary condition
\begin{equation}\label{eq:Stokescomplex3dncfem0}
0\xrightarrow{\subset} V_{h0}^g\xrightarrow{\nabla} \bs W_{h0} \xrightarrow{\curl_h} \bs V_{h0}^s \xrightarrow{\div_h}  \mathcal Q_{h0}\xrightarrow{}0
\end{equation}
is exact, where $V_{h0}^g:=V_h^g\cap H_0^1(\Omega)$, $\mathcal Q_{h0}:=\mathcal Q_h\cap L_0^2(\Omega)$, and
\[
\bs W_{h0}:=\{\bs v_h\in\bs W_h: \textrm{all the degrees of freedom \eqref{gradcurlncfmdof1}-\eqref{gradcurlncfmdof2} on $\partial\Omega$ vanish}\},
\]
\[
\bs V_{h0}^s:=\left\{\bs v_h\in\bs V_h^s: \int_F \bs v_h\dd s=\bs 0  \textrm{ for each } F\in\mathcal F_h\backslash\mathcal F_h^{i}\right\}.
\]
\end{corollary}
The space $\bs V_{h0}^s$ possesses the norm equivalence \cite[Section 10.6]{BrennerScott2008}
\begin{equation}\label{eq:normequivVsh}
\|\bs v_h\|_{1,h}\eqsim |\bs v_h|_{1,h}\quad\forall~\bs v_h\in\bs V_{h0}^s.
\end{equation}
Equip $\boldsymbol W_{h0}$ with the discrete squared norm
\[
\|\bs v_h\|_{H_h(\grad\curl)}^2:= \|\bs v_h\|_0^2+\|\curl_h\bs v_h\|_0^2+|\curl_h\bs v_h|_{1,h}^2.
\]
Since $\curl_h\bs v_h\in \bs V_{h0}^s$ for any $\bs v_h\in\boldsymbol W_{h0}$,
applying \eqref{eq:normequivVsh} to $\curl_h\bs v_h$ gives
\begin{equation*}
\|\bs v_h\|_{H_h(\grad\curl)}\eqsim \|\bs v_h\|_0+|\curl_h\bs v_h|_{1,h}\quad\forall~\bs v_h\in\boldsymbol W_{h0}.
\end{equation*}

Next we focus on the commutative diagrams for the Stokes complexes \eqref{eq:Stokescomplex3dncfem} and~\eqref{eq:Stokescomplex3dncfem0}. For this, we introduce some interpolation operators. 
For each $K\in\mathcal T_h$, let $I_K^g: H^2(K)\to \mathbb P_{k+1}(K)$ be the nodal interpolation operator of the Lagrange element \cite{Ciarlet1978},  and $\bs I_K^s: \bs H^1(K;\mathbb R^3)\to \mathbb P_{1}(K;\mathbb R^3)$ be the nodal interpolation operator of the nonconforming linear element \cite{BrennerScott2008}. We have \cite{BoffiBrezziFortin2013}
\begin{equation}\label{eq:commutingIsQlocal}
\div(\bs I_K^s\bs v)=Q_K\div\bs v\quad\forall~\bs v\in \bs H^1(K;\mathbb R^3),
\end{equation}
\begin{equation}\label{eq:estimateIKs}
\|\bs v-\bs I_K^{s}\bs v\|_{0,K}+h_K|\bs v-\bs I_K^{s}\bs v|_{1,K}\lesssim h_K^{j}|\bs v|_{j,K}\quad\forall~\bs v\in \bs H^j(K;\mathbb R^3), j=1,2.
\end{equation}
Define $\boldsymbol I_K^{gc}: \bs H^1(\curl, K)\to\bs W_k(K)$ as the nodal interpolation operator based on the degrees of freedom \eqref{gradcurlncfmdof1}-\eqref{gradcurlncfmdof2}. 
By Lemma~\ref{lem:Whunisol}, we get
\begin{equation}\label{eq:202202072}
\bs I_K^{gc}\bs q=\bs q\quad\forall~\bs q\in \bs W_k(K).
\end{equation}

\begin{lemma}
It holds
\begin{equation}\label{eq:estimateIKgc}
\|\bs v-\bs I_K^{gc}\bs v\|_{0,K}\lesssim h_K^{k+1}|\bs v|_{k+1,K}+ h_K^{2}|\bs v|_{2,K}\quad\forall~\bs v\in \bs H^2(K;\mathbb R^3).
\end{equation}
\end{lemma}
\begin{proof}
Set $\bs w=\bs v-\bs Q_K^k\bs v$ for ease of presentation.
By the norm equivalence \eqref{eq:WkKnormequivalence} and the definition of $\bs I_K^{gc}$,
\begin{align*}
\|\bs I_K^{gc}\boldsymbol w\|_{0,K}^2&\lesssim h_K^2\sum_{e\in\mathcal E(K)}\|Q_e^k((\bs I_K^{gc}\boldsymbol w)\cdot\boldsymbol t_e)\|_{0,e}^2 + h_K^3\sum_{F\in\mathcal F(K)}\|\boldsymbol{Q}_F^0(\curl(\bs I_K^{gc}\bs w)\times\bs n)\|_{0,F}^2 \\
&= h_K^2\sum_{e\in\mathcal E(K)}\|Q_e^k(\boldsymbol w\cdot\boldsymbol t_e)\|_{0,e}^2 + h_K^3\sum_{F\in\mathcal F(K)}\|\boldsymbol{Q}_F^0((\curl\bs w)\times\bs n)\|_{0,F}^2 \\
&\leq h_K^2\sum_{e\in\mathcal E(K)}\|\boldsymbol w\|_{0,e}^2 + h_K^3\sum_{F\in\mathcal F(K)}\|\curl\bs w\|_{0,F}^2.
\end{align*}
Then we obtain from \eqref{eq:202202072}, $\mathbb P_k(K;\mathbb R^3)\subset \bs W_k(K)$ and the trace inequality that
\begin{align*}
\|\bs v-\bs I_K^{gc}\bs v\|_{0,K}^2&=\|\bs w-\bs I_K^{gc}\bs w\|_{0,K}^2\leq 2\|\bs w\|_{0,K}^2+2\|\bs I_K^{gc}\bs w\|_{0,K}^2 \\
&\lesssim \|\bs w\|_{0,K}^2+h_K^2\sum_{e\in\mathcal E(K)}\|\boldsymbol w\|_{0,e}^2 + h_K^3\sum_{F\in\mathcal F(K)}\|\curl\bs w\|_{0,F}^2 \\
&\lesssim \|\bs w\|_{0,K}^2+h_K\sum_{F\in\mathcal F(K)}(\|\boldsymbol w\|_{0,F}^2+h_K^2|\boldsymbol w|_{1,F}^2+h_K^2\|\curl\bs w\|_{0,F}^2) \\
&\lesssim \|\bs w\|_{0,K}^2+h_K^2|\bs w|_{1,K}^2+h_K^4|\bs w|_{2,K}^2.
\end{align*}
Therefore the inequality \eqref{eq:estimateIKgc} holds from the error estimate of
$\bs Q_K^k$.
\end{proof}

\begin{lemma}
The operators $I_K^g$, $\bs I_K^{gc}$ and $\bs I_K^s$ satisfy the following commuting properties
\begin{equation}\label{eq:commutingIgIclocal}
\nabla(I_K^gv)=\bs I_K^{gc}(\nabla v)\quad \forall~v\in H^2(K),
\end{equation}
\begin{equation}\label{eq:commutingIcIslocal}
\curl(\bs I_K^{gc}\bs v)=\bs I_K^s(\curl \bs v)\quad \forall~\bs v\in \bs H^1(\curl, K).
\end{equation}
\end{lemma}
\begin{proof}
On each edge $e\in\mathcal E(K)$, it follows from the definitions of $I_K^g$ and $\bs I_K^{gc}$ that
\[
\int_e\left(\nabla(I_K^gv)-\bs I_K^{gc}(\nabla v)\right)\cdot\bs t_e q\dd s=\int_e\partial_{t_e}(I_K^gv-v) q\dd s=0
\quad \forall~q\in \mathbb P_k(e).
\]
On each face $F\in\mathcal F(K)$, we have
\[
\int_F\curl\left(\nabla(I_K^gv)-\bs I_K^{gc}(\nabla v)\right)\times\bs n\dd s=\int_F\curl\left(\nabla(I_K^gv-v)\right)\times\bs n\dd s =\bs0.
\]
Hence \eqref{eq:commutingIgIclocal} holds from $\nabla(I_K^gv)-\bs I_K^{gc}(\nabla v)\in \bs W_k(K)$.

On the other hand, we get from the Stokes formula that
\begin{align*}
&\int_F\left(\curl(\bs I_K^{gc}\bs v)-\bs I_K^s(\curl \bs v)\right)\cdot\bs n\dd s=\int_F\curl(\bs I_K^{gc}\bs v-\bs v)\cdot\bs n\dd s \\
=&\int_F(\bs n\times\nabla)\cdot(\bs I_K^{gc}\bs v-\bs v)\dd s=\int_F\bs t_{F,e}\cdot(\bs I_K^{gc}\bs v-\bs v)\dd s=0.
\end{align*}
And by the definitions of $\bs I_K^{gc}$ and $\bs I_K^s$,
\[
\int_F\left(\curl(\bs I_K^{gc}\bs v)-\bs I_K^s(\curl \bs v)\right)\times\bs n\dd s = \int_F\curl(\bs I_K^{gc}\bs v-\bs v)\times\bs n\dd s=\bs0.
\]
Therefore \eqref{eq:commutingIcIslocal} follows from the last two identities.
\end{proof}

Now introduce the global version of $I_K^g$, $\bs I_K^{gc}$, $\bs I_K^s$ and $Q_K$.
Let $I_h^g: H^2(\Omega)\to\boldsymbol V_h^g$, $\boldsymbol I_h^{gc}: \bs H^1(\curl, \Omega)\to\boldsymbol W_h$, $\boldsymbol I_h^s: \boldsymbol{H}^1(\Omega; \mathbb{R}^3)\to\boldsymbol V_h^s$ and $I_h^{L^2}: L^2(\Omega)\to\mathcal Q_h$ be defined by $(I_h^g v)|_K:=I_K^g(v|_K)$, $(\boldsymbol  I_h^{gc}\boldsymbol  v)|_K:=\boldsymbol  I_K^{gc}(\boldsymbol  v|_K)$, $(\boldsymbol  I_h^s\boldsymbol  v)|_K:=\boldsymbol  I_K^s(\boldsymbol  v|_K)$  and $(I_h^{L^2}v)|_K:=Q_K(v|_K)$ for each $K\in\mathcal T_h$, respectively. 
As the direct result of \eqref{eq:commutingIsQlocal}, \eqref{eq:commutingIgIclocal} and \eqref{eq:commutingIcIslocal}, we have
\begin{equation}\label{eq:commutingIgIc}
\nabla(I_h^gv)=\bs I_h^{gc}(\nabla v)\quad \forall~v\in H^2(\Omega),
\end{equation}
\begin{equation}\label{eq:commutingIcIs}
\curl_h(\bs I_h^{gc}\bs v)=\bs I_h^s(\curl \bs v)\quad \forall~\bs v\in \bs H^1(\curl, \Omega),
\end{equation}
\begin{equation}\label{eq:commutingIsQ}
\div_h(\bs I_h^s\bs v)=I_h^{L^2}\div\bs v\quad\forall~\bs v\in \bs H^1(\Omega;\mathbb R^3).
\end{equation}

Combining \eqref{eq:commutingIgIc}-\eqref{eq:commutingIsQ} and the complex \eqref{eq:Stokescomplex3dncfem} yields the commutative diagram
\begin{equation*}
\begin{array}{c}
\xymatrix{
 \mathbb R \ar[r]^-{\subset} & H^2(\Omega)\ar[d]^{I_h^g} \ar[r]^-{\nabla} & \bs H^1(\curl, \Omega) \ar[d]^{\boldsymbol{I}_h^{gc}} \ar[r]^-{\curl}
                & \boldsymbol{H}^1(\Omega; \mathbb{R}^3) \ar[d]^{\boldsymbol{I}_h^s}   \ar[r]^-{\div} & \ar[d]^{I_h^{L^2}}L^2(\Omega) \ar[r]^{} & 0 \\
 \mathbb R \ar[r]^-{\subset} & V_h^g\ar[r]^-{\nabla} & \boldsymbol W_{h} \ar[r]^{\curl_h}
                &  \boldsymbol V_{h}^s   \ar[r]^{\div_h} &  \mathcal Q_h \ar[r]^{}& 0    }
\end{array},
\end{equation*}
and the commutative diagram with homogeneous boundary conditions
\begin{equation}\label{eq:stokescdncfem0}
\resizebox{0.9\hsize}{!}{$
\begin{array}{c}
\xymatrix{
 0 \ar[r]^-{\subset} & H_0^2(\Omega)\ar[d]^{I_h^g} \ar[r]^-{\nabla} & \bs H_0^1(\curl, \Omega) \ar[d]^{\boldsymbol{I}_h^{gc}} \ar[r]^-{\curl}
                & \boldsymbol{H}_0^1(\Omega; \mathbb{R}^3) \ar[d]^{\boldsymbol{I}_h^s}   \ar[r]^-{\div} & \ar[d]^{I_h^{L^2}}L_0^2(\Omega) \ar[r]^{} & 0 \\
 0 \ar[r]^-{\subset} & V_{h0}^g\ar[r]^-{\nabla} & \boldsymbol W_{h0} \ar[r]^{\curl_h}
                &  \boldsymbol V_{h0}^s   \ar[r]^{\div_h} &  \mathcal Q_{h0} \ar[r]^{}& 0    }
\end{array}.
$}
\end{equation}

\section{Mixed finite element methods for the quad-$\curl$ problem}\label{sec:mfem}

In this section, we will advance mixed finite element methods for the quad-$\curl$ problem
\begin{equation}\label{eq:quadcurl}
\left\{
\begin{aligned}
(\curl)^4\bs u&=\bs f\quad\quad\text{in }\Omega, \\
\div\bs u&=0\quad\quad\;\text{in }\Omega, \\
\bs u\times\bs n=(\curl\bs u)\times\bs n&=\bs 0\qquad\;\text{on }\partial\Omega,
\end{aligned}
\right.
\end{equation}
where $\bs f\in \bs H(\div, \Omega)$ with $\div\bs f=0$. 
The quad-curl problem arises in inverse electromagnetic scattering theory \cite{CakoniHaddar2007} and magnetohydrodynamics \cite{ZhengHuXu2011}.

Due to the identity $\curl^2\bs v=-\Delta\bs v+\nabla(\div\bs v)$ and the fact
\[
(\curl\bs u)\cdot\bs n=(\bs n\times\nabla)\cdot\bs u=(\bs n\times\nabla)\cdot(\bs n\times\bs u\times\bs n)=0\quad\text{on }\partial\Omega,
\]
the quad-$\curl$ problem \eqref{eq:quadcurl} is equivalent to
\begin{equation}\label{eq:quadcurlgrad}
\left\{
\begin{aligned}
-\curl\Delta\curl\bs u&=\bs f\quad\quad\text{in }\Omega, \\
\div\bs u&=0\quad\quad\,\text{in }\Omega, \\
\bs u\times\bs n=\curl\bs u&=\bs 0\qquad\text{on }\partial\Omega.
\end{aligned}
\right.
\end{equation}
Then a mixed formulation of the quad-$\curl$ problem \eqref{eq:quadcurl} is to find $\bs  u\in\bs H_0(\grad\curl, \Omega)$ and $\lambda\in H_0^1(\Omega)$ such that
\begin{align}
(\nabla\curl\bs u, \nabla\curl\bs v) + (\bs v, \nabla\lambda)& =(\bs f, \bs v) \quad\forall~\bs v\in \bs H_0(\grad\curl, \Omega),  \label{eq:quadcurlmixed1} \\
(\bs u, \nabla\mu)& =0 \qquad\quad\forall~\mu\in H_0^1(\Omega).  \label{eq:quadcurlmixed2}
\end{align}

Replacing $\bs v$ in \eqref{eq:quadcurlmixed1} with $\nabla\mu$ for any $\mu\in H_0^1(\Omega)$, we get $\lambda=0$ from the fact $\div\bs f=0$. Thus it follows from \eqref{eq:quadcurlmixed1} that
\[
(\nabla\curl\bs u, \nabla\curl\bs v)  =(\bs f, \bs v) \quad\forall~\bs v\in \bs H_0(\grad\curl, \Omega).
\]

\subsection{Mixed finite element methods}

Based on the mixed formulation \eqref{eq:quadcurlmixed1}-\eqref{eq:quadcurlmixed2}, we propose the mixed finite element methods for the  quad-$\curl$ problem \eqref{eq:quadcurl} as follows: 
find $\bs  u_h\in \boldsymbol W_{h0}$ and $\lambda_h\in V_{h0}^g$ such that
\begin{align}
(\nabla_h\curl_h\bs u_h, \nabla_h\curl_h\bs v_h) + (\bs v_h, \nabla\lambda_h)& =(\bs f, \bs v_h) \quad\forall~\bs v_h\in \boldsymbol W_{h0},  \label{eq:quadcurlmfem1} \\
(\bs u_h, \nabla\mu_h)& =0 \qquad\quad\;\;\forall~\mu_h\in V_{h0}^g.  \label{eq:quadcurlmfem2}
\end{align}

Now we show the well-posedness of the mixed finite element methods \eqref{eq:quadcurlmfem1}-\eqref{eq:quadcurlmfem2} and the stability. To this end, we recall the discrete de Rham complex and the corresponding interpolation operators \cite{Arnold2018}.
Based on the degrees of freedom \eqref{Nedelecdof1}, define $\bs I_K^c: \boldsymbol{H}^2(K; \mathbb{R}^3)\to \bs V_k^c(K)$ for each $K\in\mathcal T_h$  by
\[
\int_e\bs I_K^c\bs v\cdot\bs t_e q\dd s =\int_e\bs v\cdot\bs t_e q\dd s \quad \forall~\bs v\in\boldsymbol{H}^2(K; \mathbb{R}^3),  q\in\mathbb P_k(e), e\in\mathcal E(K).
\]
Then we have
\begin{equation}\label{eq:IKc1}
\bs I_K^c\bs v=\bs v\quad \forall~\bs v\in\bs V_k^c(K),
\end{equation}
\begin{equation}\label{eq:IKc2}
\|\curl(\bs v-\bs I_K^c\bs v)\|_{0, K}\lesssim h_K|\curl\bs v|_{1, K}\quad\forall~\bs v\in\boldsymbol{H}^2(K; \mathbb{R}^3),
\end{equation}
\begin{equation}\label{eq:IKc31}
\|\bs I_K^c\bs v\|_{0, K}\lesssim \|\bs v\|_{0, K}\quad\forall~\bs v\in\bs W_k(K).
\end{equation}
Let $\bs I_h^c: \boldsymbol{H}^2(\Omega; \mathbb{R}^3)+ \bs W_h\to \bs V_h^c$ be determined by
\[
(\bs I_h^c\bs v_h)|_K:=\bs I_K^c(\bs v_h|_K)\quad\forall~K\in\mathcal T_h.
\]
The operator $\bs I_h^c$ is well-defined, since the degrees of freedom \eqref{gradcurlncfmdof1} for $\bs W_k(K)$ and \eqref{Nedelecdof1} for $\bs V_k^c(K)$ are same. And we have $\bs I_h^c\bs v_h\in \bs V_{h0}^{c}:=\bs V_{h}^{c}\cap \bs H_0(\curl, \Omega)$ when $\bs v_h\in \bs W_{h0}^{c}$. 

 Let the lowest order Raviart-Thomas element space \cite{Nedelec1980,RaviartThomas1977} 
\[
\bs V_{h0}^d:=\{\bs v_h\in \bs H_0(\div, \Omega): \bs v_h|_K\in\mathbb P_{0}(K;\mathbb R^3)+ \bs x\mathbb P_{0}(K) \textrm{ for each } K\in\mathcal T_h\}.
\]
We have the discrete de Rham complex \cite{Arnold2018}
\[
0\xrightarrow{\subset} V_{h0}^g\xrightarrow{\nabla} \bs V_{h0}^{c} \xrightarrow{\curl} \bs V_{h0}^{d} \xrightarrow{\div}  \mathcal Q_{h0}\xrightarrow{}0.
\]
Denote by $\bs I_h^d: \boldsymbol{H}_0^1(\Omega; \mathbb{R}^3)+ \bs V_{h0}^s\to \bs V_{h0}^d$ the nodal interpolation operator.
Then the commutative diagram \eqref{eq:stokescdncfem0} can be extended to
the following three-line commutative diagram
\begin{equation}\label{eq:stokesdeRhamcddiscrete}
\resizebox{0.9\hsize}{!}{$
\begin{array}{c}
\xymatrix{
 0 \ar[r]^-{\subset} & H_0^2(\Omega)\ar[d]^{I_h^g} \ar[r]^-{\nabla} & \bs H_0^1(\curl, \Omega) \ar[d]^{\boldsymbol{I}_h^{gc}} \ar[r]^-{\curl}
                & \boldsymbol{H}_0^1(\Omega; \mathbb{R}^3) \ar[d]^{\boldsymbol{I}_h^s}   \ar[r]^-{\div} & \ar[d]^{I_h^{L^2}}L_0^2(\Omega) \ar[r]^{} & 0 \\
 0 \ar[r]^-{\subset} & V_{h0}^g\ar[d]^{I}\ar[r]^-{\nabla} & \boldsymbol W_{h0} \ar[d]^{\boldsymbol{I}_h^{c}}\ar[r]^{\curl_h}
                &  \boldsymbol V_{h0}^s \ar[d]^{\boldsymbol{I}_h^d}  \ar[r]^{\div_h} & \ar[d]^{I} \mathcal Q_{h0} \ar[r]^{}& 0   \\ 
 0 \ar[r]^-{\subset} & V_{h0}^g\ar[r]^-{\nabla} & \boldsymbol V_{h0}^c \ar[r]^{\curl}
                &  \boldsymbol V_{h0}^d   \ar[r]^{\div} &  \mathcal Q_{h0} \ar[r]^{}& 0    }
\end{array},
$}
\end{equation}
where $I$ is the identity operator.

\begin{lemma}
We have
\begin{equation}\label{eq:WkKprop1}
\inf_{q\in\mathbb P_{k+1}(K)}\|\bs v-\nabla q\|_{0,K}\lesssim h_K\|\curl\bs v\|_{0,K} \quad\forall~\bs v\in\bs W_k(K), K\in\mathcal T_h.
\end{equation}
\end{lemma}
\begin{proof}
Due to \eqref{eq:poincareoperatorprop1}, $\bs v-\mathcal K_K(\curl\boldsymbol v)\in\bs W_k(K)\cap\ker(\curl)$, which means $\bs v-\mathcal K_K(\curl\boldsymbol v)\in\nabla\mathbb P_{k+1}(K)$. Choose $q\in \mathbb P_{k+1}(K)$ such that $\bs v-\mathcal K_K(\curl\boldsymbol v)=\nabla q$. Apply \eqref{eq:poincareoperatorprop2} to get
$$
\|\bs v-\nabla q\|_{0,K}=\|\mathcal K_K(\curl\boldsymbol v)\|_{0,K}\lesssim h_K\|\curl\bs v\|_{0,K},
$$
which indicates \eqref{eq:WkKprop1}.
\end{proof}

\begin{lemma}
It holds for any $K\in\mathcal T_h$ that
\begin{equation}\label{eq:IKc3}
\|\bs v-\bs I_K^c\bs v\|_{0,K}\lesssim h_K\|\curl\bs v\|_{0,K} \quad\forall~\bs v\in\bs W_k(K).
\end{equation}
\end{lemma}
\begin{proof}
Employing \eqref{eq:IKc1} and $\nabla\mathbb P_{k+1}(K)\subset \bs V_k^c(K)$, it follows
\[
\bs v-\bs I_K^c\bs v=(\bs v-\nabla q)-\bs I_K^c(\bs v-\nabla q)\quad\forall~q\in\mathbb P_{k+1}(K).
\]
Then we get from \eqref{eq:IKc31} that
$$
\|\bs v-\bs I_K^c\bs v\|_{0,K}\leq \|\bs v-\nabla q\|_{0,K}+\|\bs I_K^c(\bs v-\nabla q)\|_{0,K}\lesssim \|\bs v-\nabla q\|_{0,K},
$$
which together with the arbitrariness of $q\in\mathbb P_{k+1}(K)$ implies
\[
\|\bs v-\bs I_K^c\bs v\|_{0,K}\lesssim \inf_{q\in\mathbb P_{k+1}(K)}\|\bs v-\nabla q\|_{0,K}.
\]
Thus the inequality \eqref{eq:IKc3} follows from \eqref{eq:WkKprop1}.
\end{proof}

\begin{lemma}
It holds the discrete Poincar\'e inequality
\begin{equation}\label{eq:Ihc}
\|\bs v_h\|_{0}\lesssim \|\curl_h\bs v_h\|_{0} \quad\forall~\bs v_h\in\mathcal K_h^d,
\end{equation}
where $\mathcal K_h^d:=\{\bs v_h\in\bs W_{h0}: (\bs v_h, \nabla q_h)=0 \textrm{ for each } q_h\in V_{h0}^g\}$.
\end{lemma}
\begin{proof}
By the fact that $\bs I_h^c\bs v_h\in\bs H_0(\curl, \Omega)$, there exists $\bs\psi\in \bs H_0^1(\Omega;\mathbb R^3)$ such that (cf. \cite{GiraultRaviart1986,AmroucheBernardiDaugeGirault1998,CostabelMcIntosh2010})
\begin{equation}\label{eq:20200719-1}
\curl\bs\psi=\curl(\bs I_h^c\bs v_h), \quad \|\bs\psi\|_1\lesssim \|\curl(\bs I_h^c\bs v_h)\|_0.
\end{equation}
Let $\widetilde{\bs I}_h^c:\bs H_0(\curl, \Omega)\to \bs V_{h0}^{c}$ and $\widetilde{\bs I}_h^d:\bs H_0(\div, \Omega)\to \bs V_{h0}^{d}$ be the $L^2$ bounded projection operators devised in~\cite{ChristiansenWinther2008}.
The operators $\widetilde{\bs I}_h^c$ and $\widetilde{\bs I}_h^d$ possess the following properties
\begin{equation}\label{eq:20200719-2}
\curl(\widetilde{\bs I}_h^c\bs v)=\widetilde{\bs I}_h^d(\curl\bs v), \quad 
\|\widetilde{\bs I}_h^c\bs v\|_0\lesssim \|\bs v\|_0\quad\forall~\bs v\in \bs H_0(\curl, \Omega).
\end{equation}
By the commuting property of $\widetilde{\bs I}_h^c$ and $\widetilde{\bs I}_h^d$, it follows
\[
\curl(\widetilde{\bs I}_h^c\bs\psi)=\widetilde{\bs I}_h^d(\curl\bs\psi)=\widetilde{\bs I}_h^d(\curl(\bs I_h^c\bs v_h))=\curl(\bs I_h^c\bs v_h).
\]
By \eqref{eq:deRhamexact1}, there exists $q_h\in V_{h0}^g$ such that $\bs I_h^c\bs v_h-\widetilde{\bs I}_h^c\bs\psi=\nabla q_h$.
Because $({\bs v}_h,\nabla q_h) =0$, 
\begin{align*}
\|\bs I_h^c\bs v_h\|_0^2&=(\bs I_h^c\bs v_h, \bs I_h^c\bs v_h-\widetilde{\bs I}_h^c\bs\psi) + (\bs I_h^c\bs v_h, \widetilde{\bs I}_h^c\bs\psi) \\
&=(\bs I_h^c\bs v_h-\bs v_h, \bs I_h^c\bs v_h-\widetilde{\bs I}_h^c\bs\psi) + (\bs I_h^c\bs v_h, \widetilde{\bs I}_h^c\bs\psi).
\end{align*}
Due to \eqref{eq:20200719-2} and \eqref{eq:20200719-1}, we get
\begin{align*}
\|\bs I_h^c\bs v_h\|_0^2
&\leq\|\bs I_h^c\bs v_h-\bs v_h\|_0\|\bs I_h^c\bs v_h-\widetilde{\bs I}_h^c\bs\psi\|_0 + \|\bs I_h^c\bs v_h\|_0\|\widetilde{\bs I}_h^c\bs\psi\|_0 \\
&\lesssim\|\bs I_h^c\bs v_h-\bs v_h\|_0(\|\bs I_h^c\bs v_h\|_0+\|\bs\psi\|_1) + \|\bs I_h^c\bs v_h\|_0\|\bs\psi\|_1 \\
&\lesssim\|\bs I_h^c\bs v_h-\bs v_h\|_0(\|\bs I_h^c\bs v_h\|_0+\|\curl(\bs I_h^c\bs v_h)\|_0) + \|\bs I_h^c\bs v_h\|_0\|\|\curl(\bs I_h^c\bs v_h)\|_0 \\
&=(\|\bs v_h-\bs I_h^c\bs v_h\|_0+\|\curl(\bs I_h^c\bs v_h)\|_0)\|\bs I_h^c\bs v_h\|_0 \\
&\quad +\|\bs v_h-\bs I_h^c\bs v_h\|_0\|\curl(\bs I_h^c\bs v_h)\|_0.
\end{align*}
Thus we have
\[
\|\bs I_h^c\bs v_h\|_0\lesssim \|\bs v_h-\bs I_h^c\bs v_h\|_0 + \|\curl(\bs I_h^c\bs v_h)\|_0,
\]
which indicates
\begin{align*}
\|\bs v_h\|_0&\lesssim \|\bs v_h-\bs I_h^c\bs v_h\|_0 + \|\curl(\bs I_h^c\bs v_h)\|_0 \\
&\lesssim \|\bs v_h-\bs I_h^c\bs v_h\|_0 + \|\curl(\bs v_h-\bs I_h^c\bs v_h)\|_0 + \|\curl_h\bs v_h\|_0.
\end{align*}
Therefore \eqref{eq:Ihc} follows from \eqref{eq:IKc3}, \eqref{eq:IKc2} and the inverse inequality.
\end{proof}

\begin{lemma}
We have the discrete stability
\begin{align}
&\|\widetilde{\bs u}_h\|_{H_h(\grad\curl)}+ |\widetilde \lambda_h|_1 \notag\\
\lesssim &\sup_{(\bs v_h, \mu_h)\in\boldsymbol W_{h0}\times V_{h0}^g}\frac{(\nabla_h\curl_h\widetilde{\bs u}_h, \nabla_h\curl_h\bs v_h) + (\bs v_h, \nabla\widetilde \lambda_h)+(\widetilde{\bs u}_h, \nabla \mu_h)}{\|\bs v_h\|_{H_h(\grad\curl)}+ |\mu_h|_1} \label{eq:mfemstab}
\end{align}
for any $\widetilde{\bs u}_h\in \boldsymbol W_{h0}$ and $\widetilde{\lambda}_h\in V_{h0}^g$.
\end{lemma}
\begin{proof}
For any $\bs v_h\in\mathcal K_h^d$, by using \eqref{eq:Ihc} and \eqref{eq:normequivVsh}, we derive the coercivity 
\[
\|\bs v_h\|_{H_h(\grad\curl)}\lesssim \|\curl_h\bs v_h\|_{1,h}\lesssim |\curl_h\bs v_h|_{1,h}.
\]
Since $\nabla V_{h0}^g\subset \boldsymbol W_{h0}$, it holds the discrete inf-sup condition
\[
|\mu_h|_1=\sup_{\bs v_h\in\nabla V_{h0}^g}\frac{(\bs v_h, \nabla \mu_h)}{\|\bs v_h\|_0}=\sup_{\bs v_h\in\nabla V_{h0}^g}\frac{(\bs v_h, \nabla\mu_h)}{\|\bs v_h\|_{H_h(\grad\curl)}} \leq \sup_{\bs v_h\in\boldsymbol W_{h0}}\frac{(\bs v_h, \nabla\mu_h)}{\|\bs v_h\|_{H_h(\grad\curl)}}.
\]
Thus the discrete stability \eqref{eq:mfemstab} follows from the Babu\v{s}ka-Brezzi theory~\cite{BoffiBrezziFortin2013}.
\end{proof}

Thanks to the discrete stability \eqref{eq:mfemstab}, the mixed finite element method \eqref{eq:quadcurlmfem1}-\eqref{eq:quadcurlmfem2} are well-posed.
As the continuous case, 
replacing $\bs v_h$ in \eqref{eq:quadcurlmfem1} with $\nabla\mu_h$ for any $\mu_h\in V_{h0}^g$, we get $\lambda_h=0$ from the fact $\div\bs f=0$ again. As a result, the solution $\bs u_h\in \boldsymbol W_{h0}$ satisfies
\begin{equation}\label{eq:quadcurlmfem1reduce}
(\nabla_h\curl_h\bs u_h, \nabla_h\curl_h\bs v_h) =(\bs f, \bs v_h) \quad\forall~\bs v_h\in \boldsymbol W_{h0}.
\end{equation}

\subsection{Interpolation operator with lower regularity}
In this subsection we define an interpolation operator on $\bs H_0(\grad\curl, \Omega)$.
Since the interpolation operator $\boldsymbol I_h^{gc}$ is not well-defined on $\bs H_0(\grad\curl, \Omega)$,
we first present a regular decomposition for the space $\bs H_0(\grad\curl, \Omega)$.
\begin{lemma}
It holds the stable regular decomposition
\begin{equation}\label{eq:gradcurlregulardecomp}
\bs H_0(\grad\curl, \Omega)=\bs H_0^2(\Omega;\mathbb R^3)+\nabla H_0^1(\Omega).
\end{equation}
Specifically, for any $\bs v\in\bs H_0(\grad\curl, \Omega)$, let $\bs v_2\in \bs H_0^2(\Omega;\mathbb R^3)$ and $\boldsymbol\lambda\in\curl\bs H_0^2(\Omega;\mathbb R^3)$ satisfy
\begin{equation}\label{eq:H2stokes3d}
\left\{
\begin{aligned}
(\nabla^2\bs v_2, \nabla^2\bs\chi) + (\nabla\curl\bs\chi, \nabla\boldsymbol\lambda)& =0 \qquad\qquad\qquad\;\;\forall~\bs\chi\in \bs H_0^2(\Omega;\mathbb R^3),  \\
(\nabla\curl\bs v_2, \nabla\boldsymbol\mu)& =(\nabla\curl\bs v, \nabla\boldsymbol\mu) \quad\forall~\boldsymbol{\mu}\in \curl\bs H_0^2(\Omega;\mathbb R^3), 
\end{aligned}
\right.
\end{equation}
then there exists $v_1\in H_0^1(\Omega)$ such that
$\bs v=\bs v_2+\nabla v_1$ and
\begin{equation}\label{eq:v1v2regular}
\|\bs v_2\|_2\lesssim |\curl \bs v|_1,\quad \|v_1\|_1\lesssim \|\bs v\|_0 + \|\bs v_2\|_0.
\end{equation}
\end{lemma}
\begin{proof}
Recall the de Rham complex with homogeneous boundary conditions \cite[the second part of Theorem 1.1]{CostabelMcIntosh2010}
\[
0\xrightarrow{\subset} H_0^3(\Omega)\xrightarrow{\nabla} \bs H_0^2(\Omega;\mathbb R^3) \xrightarrow{\curl} \boldsymbol H_0^1(\Omega; \mathbb{R}^3) \xrightarrow{\div} L_0^{2}(\Omega)\xrightarrow{}0,
\]
which is exact for $\Omega$ being contractible. For $\boldsymbol{\mu}\in \curl\bs H_0^2(\Omega;\mathbb R^3)\subset\bs H_0^1(\Omega;\mathbb R^3)$, by this complex there exists $\bs w\in \bs H_0^2(\Omega;\mathbb R^3)$ satisfying
\[
\curl\bs w=\bs\mu, \quad \|\bs w\|_2\lesssim |\bs\mu|_1.
\]
Then it holds the inf-sup condition
$$
|\bs\mu|_1= \frac{(\nabla\bs\mu, \nabla\boldsymbol\mu)}{|\bs\mu|_1} \lesssim \sup_{\bs w\in\bs H_0^2(\Omega;\mathbb R^3)}\frac{(\nabla\curl\bs w, \nabla\boldsymbol\mu)}{\|\bs w\|_2}.
$$
By the Babu\v{s}ka-Brezzi theory~\cite{BoffiBrezziFortin2013}, problem \eqref{eq:H2stokes3d} is well-posed, and
$$
\curl\bs v_2=\curl\bs v, \quad \|\bs v_2\|_2 + \|\bs\lambda\|_1\lesssim |\curl \bs v|_1.
$$
Finally we finish the proof by the fact that $\curl(\bs v-\bs v_2)=\bs0$.
\end{proof}

For $\boldsymbol{v}\in\boldsymbol{H}_0(\grad\curl, \Omega)$ satisfying $\curl\boldsymbol{v}=\boldsymbol{0}$, by \eqref{eq:v1v2regular}, we have $\boldsymbol{v}_2=\boldsymbol{0}$ and $\boldsymbol{v}=\nabla v_1$.

Now we apply the regular decomposition~\eqref{eq:gradcurlregulardecomp} to $\bs v\in\bs H_0(\grad\curl, \Omega)$. 
Let $I_h^{SZ}: H_0^1(\Omega)\to V_{h0}^g$ be the Scott-Zhang interpolation operator \cite{ScottZhang1990}.
Noting that $\boldsymbol I_h^{gc}$ can be applied to $\bs v_2$,
define
$\boldsymbol \Pi_h^{gc}: \bs H_0(\grad\curl, \Omega)\to\boldsymbol W_{h0}$ as follows
\[
\boldsymbol \Pi_h^{gc}\bs v:=\boldsymbol I_h^{gc}\bs v_2+\nabla I_h^{SZ}v_1.
\]
Clearly we have
\begin{equation}\label{eq:commutingPigcIsz}
\boldsymbol \Pi_h^{gc}(\nabla v)=\nabla(I_h^{SZ}v) \quad \forall~v\in H_0^1(\Omega).
\end{equation}
We acquire from \eqref{eq:commutingIcIs} that
\begin{equation}\label{eq:commutingPigcIs}
\curl_h(\boldsymbol \Pi_h^{gc}\bs v)=\bs I_h^s(\curl \bs v)\quad \forall~\bs v\in \bs H_0(\grad\curl, \Omega).
\end{equation}

Combining \eqref{eq:commutingPigcIsz}-\eqref{eq:commutingPigcIs}, \eqref{eq:commutingIsQ} and the complex \eqref{eq:Stokescomplex3dncfem0} yields the commutative diagram
\begin{equation*}
\resizebox{0.9\hsize}{!}{$
\begin{array}{c}
\xymatrix{
 0 \ar[r]^-{\subset} & H_0^1(\Omega)\ar[d]^{I_h^{SZ}} \ar[r]^-{\nabla} & \boldsymbol H_0(\grad\curl, \Omega) \ar[d]^{\boldsymbol{\Pi}_h^{gc}} \ar[r]^-{\curl}
                & \boldsymbol{H}_0^1(\Omega; \mathbb{R}^3) \ar[d]^{\boldsymbol{I}_h^s}   \ar[r]^-{\div} & \ar[d]^{I_h^{L^2}}L_0^2(\Omega) \ar[r]^{} & 0 \\
 0 \ar[r]^-{\subset} & V_{h0}^g\ar[r]^-{\nabla} & \boldsymbol W_{h0} \ar[r]^{\curl_h}
                &  \boldsymbol V_{h0}^s   \ar[r]^{\div_h} &  \mathcal Q_{h0} \ar[r]^{}& 0    }
\end{array}.
$}
\end{equation*}

\begin{lemma}
Assume $\bs v\in\boldsymbol H^1(\Omega; \mathbb{R}^3)$ and $\curl\bs v\in\boldsymbol H^2(\Omega; \mathbb{R}^3)$. Then
\begin{equation}\label{eq:Pigchestimate1}
\|\bs v-\boldsymbol \Pi_h^{gc}\bs v\|_0\lesssim h(|\bs v|_1+|\curl \bs v|_1),
\end{equation}
\begin{equation}\label{eq:Pigchestimate2}
\|\curl_h(\bs v-\boldsymbol \Pi_h^{gc}\bs v)\|_0+ h|\curl_h(\bs v-\boldsymbol \Pi_h^{gc}\bs v)|_1\lesssim h^2|\curl\bs v|_2.
\end{equation}
\end{lemma}
\begin{proof}
Noting that $\nabla v_1=\bs v-\bs v_2$, it follows
\[
|v_1|_2\lesssim |\bs v-\bs v_2|_1\lesssim |\bs v|_1+|\curl \bs v|_1.
\]
Since $\bs v-\boldsymbol \Pi_h^{gc}\bs v=\bs v_2-\boldsymbol I_h^{gc}\bs v_2+\nabla(v_1-I_h^{SZ}v_1)$, we acquire from \eqref{eq:estimateIKgc}, the error estimate of $I_h^{SZ}$ and \eqref{eq:v1v2regular} that
\begin{align*}
\|\bs v-\boldsymbol \Pi_h^{gc}\bs v\|_0&\leq \|\bs v_2-\boldsymbol I_h^{gc}\bs v_2\|_0 + |v_1-I_h^{SZ}v_1|_1 \\
&\lesssim h^{k+1}\|\bs v_2\|_{2}+h|v_1|_2\lesssim h(|\bs v|_1+|\curl \bs v|_1).
\end{align*}
Employing \eqref{eq:estimateIKs}, we get from \eqref{eq:commutingPigcIs} that
\[
\|\curl_h(\bs v-\boldsymbol \Pi_h^{gc}\bs v)\|_0=\|\curl\bs v-\bs I_h^s(\curl\bs v)\|_0\lesssim h^2|\curl\bs v|_2,
\]
\[
|\curl_h(\bs v-\boldsymbol \Pi_h^{gc}\bs v)|_1=|\curl\bs v-\bs I_h^s(\curl\bs v)|_1\lesssim h|\curl\bs v|_2.
\]
This ends the proof.
\end{proof}


\subsection{Error analysis}
Hereafter we assume $\bs u\in\bs H_0(\grad\curl, \Omega)$ possesses the regularity $\curl\bs u\in\bs H^2(\Omega;\mathbb R^3)$, which is true for $\Omega$ being convex (See Lemma~\ref{lem:bigradcurlregularity}.).
Applying the integration by parts to the first equation in \eqref{eq:quadcurlgrad}, we derive
\begin{equation}\label{eq:quadcurluprimal1}
-(\Delta\curl\bs u, \curl\bs v) =(\bs f, \bs v) \qquad \forall~\bs v\in \bs H_0(\curl, \Omega).
\end{equation}

We first present the consistency error of the nonconforming methods \eqref{eq:quadcurlmfem1}-\eqref{eq:quadcurlmfem2}.
\begin{lemma}
We have for any $\bs v_h\in \boldsymbol W_{h0}$ that
\begin{equation}\label{eq:errorprior1}
(\nabla\curl\bs u, \nabla_h\curl_h\bs v_h) + (\Delta\curl\bs u, \curl_h\bs v_h)\lesssim h|\curl\bs u|_2|\curl_h\bs v_h|_{1,h}.
\end{equation}
\end{lemma}
\begin{proof}
Due to \eqref{eq:curlWhweakcontinuity}, we get
\begin{align*}
&\sum_{K\in\mathcal T_h}(\partial_n(\curl\bs u), \curl_h\bs v_h)_{\partial K} \\
=&\sum_{K\in\mathcal T_h}\sum_{F\in\mathcal F(K)}(\partial_n(\curl\bs u)-\bs Q_F^0\partial_n(\curl\bs u), \curl_h\bs v_h)_{F} \\
=&\sum_{K\in\mathcal T_h}\sum_{F\in\mathcal F(K)}(\partial_n(\curl\bs u)-\bs Q_F^0\partial_n(\curl\bs u), \curl_h\bs v_h-\bs Q_F^0\curl_h\bs v_h)_{F} \\
\lesssim &\, h|\curl\bs u|_2|\curl_h\bs v_h|_{1,h}.
\end{align*}
Thus \eqref{eq:errorprior1} follows from the integration by parts.
\end{proof}

\begin{lemma}
We have for any $\bs v_h\in \boldsymbol W_{h0}$ that
\begin{equation}
(\nabla\curl\bs u, \nabla_h\curl_h\bs v_h) - (\bs f, \bs v_h)\lesssim h(|\curl\bs u|_2+\|\bs f\|_0)|\curl_h\bs v_h|_{1,h}. \label{eq:errorprior2}
\end{equation}
\end{lemma}
\begin{proof}
We get from
\eqref{eq:quadcurluprimal1} with $\bs v=\bs I_h^c\bs v_h$ that
\[
 (\Delta\curl\bs u, \curl_h\bs v_h) + (\bs f, \bs v_h) 
=  (\Delta\curl\bs u, \curl_h(\bs v_h-\bs I_h^c\bs v_h)) + (\bs f, \bs v_h-\bs I_h^c\bs v_h).
\]
Applying \eqref{eq:IKc2} and \eqref{eq:IKc3} gives
\begin{align*}
&- (\Delta\curl\bs u, \curl_h\bs v_h) - (\bs f, \bs v_h) \\
\lesssim& \|\Delta\curl\bs u\|_0\|\curl_h(\bs v_h-\bs I_h^c\bs v_h)\|_{0} + \|\bs f\|_0\|\bs v_h-\bs I_h^c\bs v_h\|_{0} \\
\lesssim& h|\curl\bs u|_2|\curl_h\bs v_h|_{1,h} + h\|\bs f\|_0\|\curl_h\bs v_h\|_{0}.
\end{align*}
Together with \eqref{eq:errorprior1}, we derive
\begin{align*}
&(\nabla\curl\bs u, \nabla_h\curl_h\bs v_h) - (\bs f, \bs v_h) \\
\lesssim& h|\curl\bs u|_2|\curl_h\bs v_h|_{1,h} + h\|\bs f\|_0\|\curl_h\bs v_h\|_{0}.
\end{align*}
Hence \eqref{eq:errorprior2} follows from \eqref{eq:normequivVsh}.
\end{proof}

Now we can show the a priori error estimate.

\begin{theorem}
Let $\bs  u\in\bs H_0(\grad\curl, \Omega)$ be the solution of the problem \eqref{eq:quadcurl}, and $\bs u_h\in \boldsymbol W_{h0}$ the solution of the mixed finite element methods \eqref{eq:quadcurlmfem1}-\eqref{eq:quadcurlmfem2}.
Assume $\bs u\in\boldsymbol H^1(\Omega; \mathbb{R}^3)$ and $\curl\bs u\in\boldsymbol H^2(\Omega; \mathbb{R}^3)$. 
We have
\begin{equation}\label{eq:errorprioruh}
\|\bs u-\bs u_h\|_{H_h(\grad\curl)}\lesssim h(\|\curl\bs u\|_2+|\bs u|_1+\|\bs f\|_0).
\end{equation}
\end{theorem}
\begin{proof}
It follows from \eqref{eq:errorprior2} that
\begin{align}
&(\nabla_h\curl_h(\boldsymbol \Pi_h^{gc}\bs u), \nabla_h\curl_h\bs v_h)-(\bs f, \bs v_h) \notag\\
=&(\nabla_h\curl_h(\boldsymbol \Pi_h^{gc}\bs u-\bs u), \nabla_h\curl_h\bs v_h) + (\nabla\curl\bs u, \nabla_h\curl_h\bs v_h)-(\bs f, \bs v_h) \notag\\
\lesssim& |\curl_h(\boldsymbol \Pi_h^{gc}\bs u-\bs u)|_{1,h}|\curl_h\bs v_h|_{1,h}+h(|\curl\bs u|_2+\|\bs f\|_0)|\curl_h\bs v_h|_{1,h}. \label{eq:20200719-3}
\end{align}

On the other hand, by the discrete stability \eqref{eq:mfemstab} with $\widetilde{\bs u}_h=\boldsymbol \Pi_h^{gc}\bs u-\bs u_h$
and $\widetilde\lambda_h=0$, we get from \eqref{eq:quadcurlmfem1reduce} and the fact $\bs u_h\in \mathcal K_h^d$ that
\begin{align*}
&\|\boldsymbol \Pi_h^{gc}\bs u-\bs u_h\|_{H_h(\grad\curl)} \\
\lesssim&\sup_{(\bs v_h, \mu_h)\in\boldsymbol W_{h0}\times V_{h0}^g}\frac{(\nabla_h\curl_h(\boldsymbol \Pi_h^{gc}\bs u-\bs u_h), \nabla_h\curl_h\bs v_h) +(\boldsymbol \Pi_h^{gc}\bs u-\bs u_h, \nabla\mu_h)}{\|\bs v_h\|_{H_h(\grad\curl)}+ |\mu_h|_1}  \\
=&\sup_{(\bs v_h, \mu_h)\in\boldsymbol W_{h0}\times V_{h0}^g}\frac{(\nabla_h\curl_h(\boldsymbol \Pi_h^{gc}\bs u), \nabla_h\curl_h\bs v_h)-(\bs f, \bs v_h) +(\boldsymbol \Pi_h^{gc}\bs u-\bs u, \nabla\mu_h)}{\|\bs v_h\|_{H_h(\grad\curl)}+ |\mu_h|_1} \\
\lesssim& \|\bs u-\boldsymbol \Pi_h^{gc}\bs u\|_0 + \sup_{\bs v_h\in\boldsymbol W_{h0}}\frac{(\nabla_h\curl_h(\boldsymbol \Pi_h^{gc}\bs u), \nabla_h\curl_h\bs v_h)-(\bs f, \bs v_h)}{\|\bs v_h\|_{H_h(\grad\curl)}}. 
\end{align*}
Hence we get from \eqref{eq:20200719-3} that
\[
\|\boldsymbol \Pi_h^{gc}\bs u-\bs u_h\|_{H_h(\grad\curl)}\lesssim \|\bs u-\boldsymbol \Pi_h^{gc}\bs u\|_{H_h(\grad\curl)}+h(|\curl\bs u|_2+\|\bs f\|_0).
\]
Thus
\begin{align*}
\|\bs u-\bs u_h\|_{H_h(\grad\curl)}\lesssim&\|\bs u-\boldsymbol \Pi_h^{gc}\bs u\|_{H_h(\grad\curl)} + \|\boldsymbol \Pi_h^{gc}\bs u-\bs u_h\|_{H_h(\grad\curl)} \\
\lesssim & \|\bs u-\boldsymbol \Pi_h^{gc}\bs u\|_{H_h(\grad\curl)} + h(|\curl\bs u|_2+\|\bs f\|_0).
\end{align*}
Finally \eqref{eq:errorprioruh} follows from \eqref{eq:Pigchestimate1} and \eqref{eq:Pigchestimate2}.
\end{proof}

\begin{remark}\rm
As illustrated in \cite[Section 7.9]{BoffiBrezziFortin2013},
the convergence would deteriorate if using the nonconforming finite element space $\boldsymbol W_{h0}$ to discretize Maxwell equation
\begin{equation*}
\left\{
\begin{aligned}
\curl\curl\bs u&=\bs f\quad\quad\text{in }\Omega, \\
\div\bs u&=0\quad\quad\,\text{in }\Omega, \\
\bs u\times\bs n&=\bs 0\qquad\text{on }\partial\Omega.
\end{aligned}
\right.
\end{equation*}
\end{remark}

Next we estimate $\|\curl_h(\bs u-\bs u_h)\|_{0}$ by the duality argument. To this end, consider the dual problem
\begin{equation}\label{eq:quadcurldual}
\left\{
\begin{aligned}
-\curl\Delta\curl\widetilde{\bs u}&=\curl\curl\bs I_h^c(\boldsymbol \Pi_h^{gc}\bs u - \bs u_h) \quad\;\;\;\text{in }\Omega, \\
\div\widetilde{\bs u}&=0\qquad\qquad\qquad\qquad\qquad\quad\;\;\,\text{in }\Omega, \\
\widetilde{\bs u}\times\bs n=(\curl\widetilde{\bs u})\times\bs n&=\bs 0\qquad\qquad\qquad\qquad\qquad\quad\;\;\text{on }\partial\Omega,
\end{aligned}
\right.
\end{equation}
where $\widetilde{\bs u}\in\bs H_0(\grad\curl, \Omega)$.
The first equation in the dual problem \eqref{eq:quadcurldual} holds in the sense of $\bs H^{-1}(\div, \Omega)$, where
$\bs H^{-1}(\div, \Omega):=\{\bs v\in\bs H^{-1}(\Omega; \mathbb R^3): \div\bs v\in H^{-1}(\Omega)\}$ is the dual space of $\bs H_0(\curl, \Omega)$ \cite{ChenHuang2018}. Thanks to \eqref{eq:stokesdeRhamcddiscrete} and \eqref{eq:commutingPigcIs}, it holds
\begin{align}
\curl\bs I_h^c(\boldsymbol \Pi_h^{gc}\bs u - \bs u_h)&=\bs I_h^d\curl_h(\boldsymbol \Pi_h^{gc}\bs u - \bs u_h)=\bs I_h^d(\bs I_h^s\curl\bs u - \curl_h\bs u_h) \notag\\
&=\bs I_h^d\curl_h(\bs u - \bs u_h). \label{eq:dualrhs1}
\end{align}

We assume the dual problem \eqref{eq:quadcurldual} possesses the following regularity in this section
\begin{equation}\label{eq:quadcurldualregular}
\|\widetilde{\bs u}\|_1+\|\curl\widetilde{\bs u}\|_2\lesssim \|\curl\curl\bs I_h^c(\boldsymbol \Pi_h^{gc}\bs u - \bs u_h) \|_{-1}\leq \|\bs I_h^d\curl_h(\bs u - \bs u_h)\|_0.
\end{equation}
The regularity \eqref{eq:quadcurldualregular} holds for the domain $\Omega$ being convex (See Lemma~\ref{lem:bigradcurlregularity}.).
Similarly as \eqref{eq:quadcurluprimal1},  it holds from \eqref{eq:quadcurldual} that
\begin{equation}\label{eq:20200815-3}
-(\Delta\curl\widetilde{\bs u}, \curl\bs v) =(\curl\bs I_h^c(\boldsymbol \Pi_h^{gc}\bs u - \bs u_h), \curl\bs v) \qquad \forall~\bs v\in \bs H_0(\curl, \Omega).
\end{equation}

\begin{theorem}
Let $\bs  u\in\bs H_0(\grad\curl, \Omega)$ be the solution of the problem \eqref{eq:quadcurl}, and $\bs u_h\in \boldsymbol W_{h0}$ the solution of the mixed finite element methods \eqref{eq:quadcurlmfem1}-\eqref{eq:quadcurlmfem2}.
Assume the regularity \eqref{eq:quadcurldualregular} holds. 
We have
\begin{equation}\label{eq:errorprioruhcurl}
\|\curl_h(\bs u - \bs u_h)\|_0\lesssim  h^{k+1}\|\bs f\|_0+h^2(\|\curl\bs u\|_2+|\bs u|_1).
\end{equation}
\end{theorem}
\begin{proof}
It follows from
\eqref{eq:estimateIKs} that
\begin{align*}
&\sum_{K\in\mathcal T_h}(\partial_n(\curl\bs u), \boldsymbol I_h^{s}\curl\widetilde{\bs u}-\curl\widetilde{\bs u})_{\partial K} \\
=&\sum_{K\in\mathcal T_h}\sum_{F\in\mathcal F(K)}(\partial_n(\curl\bs u)-\bs Q_F^0\partial_n(\curl\bs u), \boldsymbol I_h^{s}\curl\widetilde{\bs u}-\curl\widetilde{\bs u})_{F} \\
\lesssim & h^2|\curl\bs u|_2|\curl\widetilde{\bs u}|_2.
\end{align*}
Applying \eqref{eq:estimateIKs} again, we get
\begin{align*}
&(\nabla\curl\bs u, \nabla_h(\boldsymbol I_h^{s}\curl\widetilde{\bs u}-\curl\widetilde{\bs u})) \\
=&\sum_{K\in\mathcal T_h}(\partial_n(\curl\bs u), \boldsymbol I_h^{s}\curl\widetilde{\bs u}-\curl\widetilde{\bs u})_{\partial K}-(\Delta\curl\bs u, \boldsymbol I_h^{s}\curl\widetilde{\bs u}-\curl\widetilde{\bs u}) \\
\lesssim & h^2|\curl\bs u|_2|\curl\widetilde{\bs u}|_2.
\end{align*}
Due to \eqref{eq:estimateIKgc} and the fact $\div\bs f=0$, we have
\[
(\bs f, \widetilde{\bs u}-\boldsymbol \Pi_h^{gc}\widetilde{\bs u}) =(\bs f, \widetilde{\bs u}_2-\boldsymbol I_h^{gc}\widetilde{\bs u}_2) \lesssim h^{k+1}\|\bs f\|_0\|\widetilde{\bs u}_2\|_2\lesssim h^{k+1}\|\bs f\|_0 |\curl \widetilde{\bs u}|_1.
\]
Combining the last two inequalities, \eqref{eq:quadcurlmfem1reduce} and \eqref{eq:commutingPigcIs} implies
\begin{align*}
&\quad\; (\nabla_h\curl_h(\bs u - \bs u_h), \nabla_h\curl_h\boldsymbol \Pi_h^{gc}\widetilde{\bs u})\\
&=(\nabla\curl\bs u, \nabla_h\boldsymbol I_h^{s}\curl\widetilde{\bs u})- (\bs f, \boldsymbol \Pi_h^{gc}\widetilde{\bs u}) \\
&=(\nabla\curl\bs u, \nabla_h(\boldsymbol I_h^{s}\curl\widetilde{\bs u}-\curl\widetilde{\bs u}))+ (\bs f, \widetilde{\bs u}-\boldsymbol \Pi_h^{gc}\widetilde{\bs u}) \\
&\lesssim h^2|\curl\bs u|_2|\curl\widetilde{\bs u}|_2+ h^{k+1}\|\bs f\|_0 |\curl \widetilde{\bs u}|_1.
\end{align*}
Employing \eqref{eq:commutingPigcIs} and \eqref{eq:estimateIKs}, we get
\begin{align*}
&\quad(\nabla_h\curl_h(\bs u - \bs u_h), \nabla_h\curl_h(\widetilde{\bs u}-\boldsymbol \Pi_h^{gc}\widetilde{\bs u})) \\
&=(\nabla_h\curl_h(\bs u - \bs u_h), \nabla_h(\curl\widetilde{\bs u}-\boldsymbol I_h^{s}\curl\widetilde{\bs u})) \\
&\leq |\curl_h(\bs u - \bs u_h)|_{1,h}|\curl\widetilde{\bs u}-\boldsymbol I_h^{s}\curl\widetilde{\bs u}|_{1,h}\lesssim h |\curl_h(\bs u - \bs u_h)|_{1,h}|\curl\widetilde{\bs u}|_{2}.
\end{align*}
It holds from the sum of the last two inequalities that
\begin{align*}
&(\nabla_h\curl_h(\bs u - \bs u_h), \nabla\curl\widetilde{\bs u}) \\
\lesssim &(h^2|\curl\bs u|_2+h |\curl_h(\bs u - \bs u_h)|_{1,h}+ h^{k+1}\|\bs f\|_0 )\|\curl\widetilde{\bs u}\|_{2}.
\end{align*}
Thanks to \eqref{eq:curlWhweakcontinuity}, we obtain
\begin{align*}
&-\sum_{K\in\mathcal T_h}(\curl_h(\bs u - \bs u_h), \partial_n\curl\widetilde{\bs u})_{\partial K} \\
=&-\sum_{K\in\mathcal T_h}\sum_{F\in\mathcal F(K)}((\bs I-\bs Q_{F}^0)\curl_h(\bs u - \bs u_h), (\bs I-\bs Q_{F}^0)\partial_n\curl\widetilde{\bs u})_{\partial K} \\
\lesssim &\, h |\curl_h(\bs u - \bs u_h)|_{1,h}|\curl\widetilde{\bs u}|_{2}.
\end{align*}
Hence we achieve from the last two inequalities that
\begin{align*}
& -(\curl_h(\bs u - \bs u_h), \Delta\curl\widetilde{\bs u}) \\
\lesssim &(h^2|\curl\bs u|_2+h |\curl_h(\bs u - \bs u_h)|_{1,h}+ h^{k+1}\|\bs f\|_0 )\|\curl\widetilde{\bs u}\|_{2}.
\end{align*}
On the other hand, it follows from \eqref{eq:dualrhs1} and\eqref{eq:20200815-3} that
\begin{align*}
&\quad\;\|\bs I_h^d\curl_h(\bs u - \bs u_h)\|_0^2=-(\bs I_h^d\curl_h(\bs u - \bs u_h), \Delta\curl\widetilde{\bs u}) \\
&=((\bs I-\bs I_h^d)\curl_h(\bs u - \bs u_h), \Delta\curl\widetilde{\bs u})-(\curl_h(\bs u - \bs u_h), \Delta\curl\widetilde{\bs u}) \\
&\lesssim (h^2|\curl\bs u|_2+h |\curl_h(\bs u - \bs u_h)|_{1,h}+ h^{k+1}\|\bs f\|_0 )\|\curl\widetilde{\bs u}\|_{2},
\end{align*}
which together with \eqref{eq:quadcurldualregular} yields
\[
\|\bs I_h^d\curl_h(\bs u - \bs u_h)\|_0\lesssim h^2|\curl\bs u|_2+h |\curl_h(\bs u - \bs u_h)|_{1,h}+ h^{k+1}\|\bs f\|_0 .
\]
Hence
\[
\|\curl_h(\bs u - \bs u_h)\|_0\lesssim h^2|\curl\bs u|_2+h |\curl_h(\bs u - \bs u_h)|_{1,h}+ h^{k+1}\|\bs f\|_0 .
\]
Finally \eqref{eq:errorprioruhcurl} follows from \eqref{eq:errorprioruh}.
\end{proof}

\section{Decoupling of the mixed finite element methods}\label{sec:decoupling}
In this section, we will present an equivalent decoupled discretization of the mixed finite element methods \eqref{eq:quadcurlmfem1}-\eqref{eq:quadcurlmfem2} as the decoupled Morley element method for the biharmonic equation in \cite{Huang2010}, based on which a fast solver is suggested.    

\subsection{Decoupling}
In the continuous level,
the mixed formulation \eqref{eq:quadcurlmixed1}-\eqref{eq:quadcurlmixed2} of the quad-$\curl$ problem \eqref{eq:quadcurlgrad} can be decoupled into the following system \cite{ChenHuang2018,Zhang2018a}:
find $\bs w\in \bs H_0(\curl, \Omega)$, $\lambda\in H_0^{1}(\Omega)$, $\boldsymbol{\phi}\in \boldsymbol H_0^1(\Omega; \mathbb{R}^3)$, $p\in L_0^{2}(\Omega)$, $\bs u\in \bs H_0(\curl, \Omega)$ and $\sigma\in H_0^{1}(\Omega)$ such that
\begin{align*}
(\curl  \bs w, \curl  \bs v)+(\bs v,\nabla \lambda)&=(\bs f, \bs v)  \quad\quad\quad\quad \forall~\bs v\in \bs H_0(\curl, \Omega),  \\ 
(\bs w,\nabla \tau) &= 0 \quad\quad\quad\quad\quad\;\;\;\, \forall~ \tau\in H_0^{1}(\Omega),  \\ 
(\boldsymbol\nabla\boldsymbol\phi, \boldsymbol\nabla\boldsymbol\psi) + (\div\boldsymbol\psi, p) & =(\curl  \bs w, \boldsymbol\psi) \quad\;\;\, \forall~\boldsymbol{\psi}\in  \boldsymbol H_0^1(\Omega; \mathbb{R}^3), \\ 
(\div\boldsymbol\phi, q) &= 0 \quad\quad\quad\quad\quad\;\;\;\, \forall~ q\in L_0^{2}(\Omega),  \\ 
(\curl \bs  u, \curl \bs  \chi)+(\bs \chi,\nabla\sigma)&= (\boldsymbol{\phi}, \curl \bs \chi) \quad\quad \forall~\bs \chi\in \bs H_0(\curl, \Omega), \\ 
(\bs u,\nabla \mu) &= 0 \quad\quad\quad\quad\quad\;\;\;\, \forall~ \mu\in H_0^{1}(\Omega).  
\end{align*}

Thanks to the discrete Stokes complex \eqref{eq:Stokescomplex3dncfem0}, the mixed finite element methods \eqref{eq:quadcurlmfem1}-\eqref{eq:quadcurlmfem2} can also be decoupled to
find $\bs w_h\in \bs W_{h0}$, $\lambda_h\in V_{h0}^g$, $\boldsymbol{\phi}_h\in \bs V_{h0}^s$, $p_h\in\mathcal Q_{h0}$, $\bs u_h\in \bs W_{h0}$ and $\sigma_h\in V_{h0}^g$ such that
\begin{align}
(\curl_h\bs w_h, \curl_h\bs v_h)+(\bs v_h,\nabla \lambda_h)&=(\bs f, \bs v_h)  \qquad\quad\quad\;\; \forall~\bs v_h\in \bs W_{h0},  \label{eq:quarticcurlfemdecouple1}\\
(\bs w_h,\nabla \tau_h) &= 0 \quad\quad\quad\quad\quad\quad\;\;\;\, \forall~ \tau_h\in V_{h0}^g,  \label{eq:quarticcurlfemdecouple2}\\
(\boldsymbol\nabla_h\boldsymbol\phi_h, \boldsymbol\nabla_h\boldsymbol\psi_h) + (\div_h\boldsymbol\psi_h, p_h) & =(\curl_h\bs w_h, \boldsymbol\psi_h) \quad\; \forall~\boldsymbol{\psi}_h\in\bs V_{h0}^s, \label{eq:quarticcurlfemdecouple3}\\
(\div_h\boldsymbol\phi_h, q_h) &= 0 \quad\quad\quad\quad\quad\quad\quad \forall~ q_h\in\mathcal Q_{h0},  \label{eq:quarticcurlfemdecouple4}\\
(\curl_h\bs  u_h, \curl_h\bs \chi_h)+(\bs \chi_h,\nabla\sigma_h)&= (\boldsymbol{\phi}_h, \curl_h\bs \chi_h) \quad\;\; \forall~\bs \chi_h\in \bs W_{h0}, \label{eq:quarticcurlfemdecouple5} \\
(\bs u_h,\nabla \mu_h) &= 0 \quad\quad\quad\quad\quad\quad\quad \forall~ \mu_h\in V_{h0}^g.  \label{eq:quarticcurlfemdecouple6}
\end{align}
Both \eqref{eq:quarticcurlfemdecouple1}-\eqref{eq:quarticcurlfemdecouple2} and \eqref{eq:quarticcurlfemdecouple5}-\eqref{eq:quarticcurlfemdecouple6}  are mixed finite element methods for the Maxwell equation. From the discrete Poincar\'e inequality \eqref{eq:Ihc} and the fact $\nabla V_{h0}^g\subset\bs W_{h0}^g$, we have the discrete stability
\[
\|\widetilde{\bs w}_h\|_{H_h(\curl)}+ |\widetilde\lambda_h|_1 \lesssim \sup_{(\bs v_h, \tau_h)\in\boldsymbol W_{h0}\times V_{h0}^g}\frac{(\curl_h\widetilde{\bs w}_h, \curl_h\bs v_h)+(\bs v_h,\nabla\widetilde\lambda_h) + (\widetilde{\bs w}_h,\nabla \tau_h)}{\|\bs v_h\|_{H_h(\curl)}+ |\tau_h|_1} 
\]
for any $\widetilde{\bs w}_h\in \boldsymbol W_{h0}$ and $\widetilde{\lambda}_h\in V_{h0}^g$, where the squared norm
\[
\|\bs v_h\|_{H_h(\curl)}^2:= \|\bs v_h\|_{0}^2 + \|\curl_h\bs v_h\|_{0}^2.
\]
Hence both mixed finite element methods \eqref{eq:quarticcurlfemdecouple1}-\eqref{eq:quarticcurlfemdecouple2} and \eqref{eq:quarticcurlfemdecouple5}-\eqref{eq:quarticcurlfemdecouple6} are well-posed.
The discrete method \eqref{eq:quarticcurlfemdecouple3}-\eqref{eq:quarticcurlfemdecouple4} is exactly the nonconforming $P_1$-$P_0$ element method for the Stokes equation.

By replacing $\bs v_h$ in \eqref{eq:quarticcurlfemdecouple1} with $\nabla\mu_h$ for any $\mu_h\in V_{h0}^g$, we obtain $\lambda_h=0$ from the fact $\div\bs f=0$. Similarily we achieve $\sigma_h=0$ from \eqref{eq:quarticcurlfemdecouple5}. Then \eqref{eq:quarticcurlfemdecouple1} and \eqref{eq:quarticcurlfemdecouple5} will be reduced to
\begin{equation}\label{eq:quarticcurlfemdecouple1reduce}
(\curl_h\bs w_h, \curl_h\bs v_h)=(\bs f, \bs v_h)  \quad \forall~\bs v_h\in \bs W_{h0},  
\end{equation}
and
\begin{equation}\label{eq:quarticcurlfemdecouple5reduce}
(\curl_h\bs  u_h, \curl_h\bs \chi_h)= (\boldsymbol{\phi}_h, \curl_h\bs \chi_h) \quad\forall~\bs \chi_h\in \bs W_{h0}. 
\end{equation}

\begin{theorem}
Let $(\bs w_h, 0, \boldsymbol{\phi}_h, p_h, \bs u_h, 0)\in \bs W_{h0}\times V_{h0}^g\times\bs V_{h0}^s\times \mathcal Q_{h0}\times\bs W_{h0}\times V_{h0}^g$ be the solution of the discrete methods \eqref{eq:quarticcurlfemdecouple1}-\eqref{eq:quarticcurlfemdecouple6}. Then $(\bs u_h, 0)$ is the solution of the mixed finite element methods \eqref{eq:quadcurlmfem1}-\eqref{eq:quadcurlmfem2}.
\end{theorem}
\begin{proof}
Since \eqref{eq:quarticcurlfemdecouple6} and \eqref{eq:quadcurlmfem2} are same, we only have to show $\bs u_h\in \mathcal K_h^d$ satisfies~\eqref{eq:quadcurlmfem1reduce}.
It follows from \eqref{eq:quarticcurlfemdecouple4} and the complex~\eqref{eq:Stokescomplex3dncfem0} that there exists $\widetilde{\bs u}_h\in \mathcal K_h^d$ satisfying $\boldsymbol\phi_h=\curl_h\widetilde{\bs u}_h$, which together with \eqref{eq:quarticcurlfemdecouple5reduce} yields
\[
(\curl_h(\bs u_h-\widetilde{\bs u}_h), \curl_h\bs \chi_h)= 0 \quad\forall~\bs \chi_h\in \bs W_{h0}. 
\]
Hence $\widetilde{\bs u}_h=\bs u_h$ and $\boldsymbol\phi_h=\curl_h\bs u_h$. Taking $\bs \psi_h=\curl_h\bs v_h$ in \eqref{eq:quarticcurlfemdecouple3} with $\bs v_h\in\bs W_{h0}$, we derive from~\eqref{eq:quarticcurlfemdecouple1reduce} that
\[
(\boldsymbol\nabla_h\curl_h\bs u_h, \boldsymbol\nabla_h\curl_h\bs v_h)  =(\curl_h\bs w_h, \curl_h\bs v_h)=(\bs f, \bs v_h).
\]
Thus the discrete methods \eqref{eq:quarticcurlfemdecouple1}-\eqref{eq:quarticcurlfemdecouple6} and the mixed methods \eqref{eq:quadcurlmfem1}-\eqref{eq:quadcurlmfem2} are equivalent.
\end{proof}

\subsection{A fast solver}
We discuss a fast solver for the mixed methods \eqref{eq:quadcurlmfem1}-\eqref{eq:quadcurlmfem2} in this subsection.
The equivalence between the mixed methods \eqref{eq:quadcurlmfem1}-\eqref{eq:quadcurlmfem2} and
the mixed methods~\eqref{eq:quarticcurlfemdecouple1}-\eqref{eq:quarticcurlfemdecouple6} suggests fast solvers for the mixed finite element methods \eqref{eq:quadcurlmfem1}-\eqref{eq:quadcurlmfem2}.
We can solve the mixed method \eqref{eq:quarticcurlfemdecouple1}-\eqref{eq:quarticcurlfemdecouple2}, the mixed method \eqref{eq:quarticcurlfemdecouple3}-\eqref{eq:quarticcurlfemdecouple4} and the mixed method \eqref{eq:quarticcurlfemdecouple5}-\eqref{eq:quarticcurlfemdecouple6} sequentially.
The mixed methods \eqref{eq:quarticcurlfemdecouple1}-\eqref{eq:quarticcurlfemdecouple2} and~\eqref{eq:quarticcurlfemdecouple5}-\eqref{eq:quarticcurlfemdecouple6} for the Maxwell equation can be efficiently solved by the solver in
\cite[Section 4.4]{ChenWuZhongZhou2018}. And for the the mixed method \eqref{eq:quarticcurlfemdecouple3}-\eqref{eq:quarticcurlfemdecouple4} of the Stokes equation, we can adopt the block diagonal preconditioner \cite{ElmanSilvesterWathen2005} or the approximate block-factorization preconditioner \cite{ChenHuHuang2018}.

Finally we demonstrate the fast solver for the mixed methods~\eqref{eq:quarticcurlfemdecouple1}-\eqref{eq:quarticcurlfemdecouple2} and~\eqref{eq:quarticcurlfemdecouple5}-\eqref{eq:quarticcurlfemdecouple6}.
To this end, define the inner product
\[
\langle \lambda_h, \mu_h\rangle:=\sum_{i=1}^{n_g}\lambda_i\mu_i\|\psi_i\|_0^2, \textrm{ where } \lambda_h=\sum_{i=1}^{n_g}\lambda_i\psi_i, \mu_h=\sum_{i=1}^{n_g}\mu_i\psi_i
\]
with $\{\psi_i\}_1^{n_g}$ being the basis functions of $V_{h0}^g$. 
The matrix of $\langle\lambda_h, \mu_h\rangle$ is just the diagonal of the mass matrix of $(\lambda_h, \mu_h)$.
Then we introduce the following two mixed methods
\begin{align}
(\curl_h\bs w_h, \curl_h\bs v_h)+(\bs v_h,\nabla \lambda_h)&=(\bs f, \bs v_h)  \qquad\quad\quad\;\; \forall~\bs v_h\in \bs W_{h0},  \label{eq:quarticcurlfemdecouple1equiv}\\
(\bs w_h,\nabla \tau_h) - \langle\lambda_h, \tau_h\rangle&= 0 \quad\quad\quad\quad\quad\quad\;\;\;\, \forall~ \tau_h\in V_{h0}^g,  \label{eq:quarticcurlfemdecouple2equiv}
\end{align}
and
\begin{align}
(\curl_h\bs  u_h, \curl_h\bs \chi_h)+(\bs \chi_h,\nabla\sigma_h)&= (\boldsymbol{\phi}_h, \curl_h\bs \chi_h) \quad\;\; \forall~\bs \chi_h\in \bs W_{h0}, \label{eq:quarticcurlfemdecouple5equiv} \\
(\bs u_h,\nabla \mu_h) - \langle\sigma_h, \mu_h\rangle &= 0 \quad\quad\quad\quad\quad\quad\quad \forall~ \mu_h\in V_{h0}^g.  \label{eq:quarticcurlfemdecouple6equiv}
\end{align}
The well-posedness of the mixed methods \eqref{eq:quarticcurlfemdecouple1equiv}-\eqref{eq:quarticcurlfemdecouple2equiv} and \eqref{eq:quarticcurlfemdecouple5equiv}-\eqref{eq:quarticcurlfemdecouple6equiv} follows from the stability of the mixed methods \eqref{eq:quarticcurlfemdecouple1}-\eqref{eq:quarticcurlfemdecouple2} and \eqref{eq:quarticcurlfemdecouple5}-\eqref{eq:quarticcurlfemdecouple6}.

\begin{lemma}
The mixed method \eqref{eq:quarticcurlfemdecouple1equiv}-\eqref{eq:quarticcurlfemdecouple2equiv} is equivalent to the mixed method \eqref{eq:quarticcurlfemdecouple1}-\eqref{eq:quarticcurlfemdecouple2}.
And
the mixed method \eqref{eq:quarticcurlfemdecouple5equiv}-\eqref{eq:quarticcurlfemdecouple6equiv} is equivalent to the mixed method  \eqref{eq:quarticcurlfemdecouple5}-\eqref{eq:quarticcurlfemdecouple6}.
\end{lemma}
\begin{proof}
Suppose $(\bs w_h, 0)\in \bs W_{h0}\times V_{h0}^g$ is the solution of the mixed method \eqref{eq:quarticcurlfemdecouple1}-\eqref{eq:quarticcurlfemdecouple2}.  By the fact $\lambda_h=0$, apparently $(\bs w_h, 0)$ is also the solution of the mixed method \eqref{eq:quarticcurlfemdecouple1equiv}-\eqref{eq:quarticcurlfemdecouple2equiv}. The equivalence between the mixed method \eqref{eq:quarticcurlfemdecouple5equiv}-\eqref{eq:quarticcurlfemdecouple6equiv} and the mixed method  \eqref{eq:quarticcurlfemdecouple5}-\eqref{eq:quarticcurlfemdecouple6} follows similarly.
\end{proof}

Such equivalence in matrix form has been revealed in \cite[(77)-(78)]{ChenWuZhongZhou2018}.
The matrix form of the mixed finite element method \eqref{eq:quarticcurlfemdecouple1equiv}-\eqref{eq:quarticcurlfemdecouple2equiv} is 
\[
\begin{pmatrix}
A & B^{\intercal} \\
B & -D
\end{pmatrix}\begin{pmatrix}
\bs w_h \\
\lambda_h
\end{pmatrix}=\begin{pmatrix}
\bs f \\
0
\end{pmatrix}.
\]
Here we still use $\bs w_h$, $\lambda_h$ and $\bs f$ to represent the vector forms of $\bs w_h$, $\lambda_h$ and $(\bs f, \bs v_h)$ for ease of presentation. Noting that $D$ is diagonal, we get
\[
(A+B^{\intercal}D^{-1}B)\bs w_h=\bs f. 
\]
The Schur complement $A+B^{\intercal}D^{-1}B$ corresponds to the symmetric matrix of a discontinuous Galerkin method for the vector Laplacian, which is positive definite and can be solved by the conjugate gradient method with the HX preconditioner in \cite{HiptmairXu2007}.


\section{Numerical results}\label{sec:numerresults}
In this section, we perform a numerical experiment to demonstrate the theoretical results of the mixed finite element methods \eqref{eq:quadcurlmfem1}-\eqref{eq:quadcurlmfem2}. Let $\Omega=(0,1)^3$. And choose the function $\bs f$ in \eqref{eq:quadcurl} such that the exact solution of \eqref{eq:quadcurl} is
\[
\bs u = \curl\begin{pmatrix}
0\\
0\\
\sin^3(\pi x)\sin^3(\pi y)\sin^3(\pi z)
\end{pmatrix}.
\]
We take uniform triangulations on $\Omega$. Set $k=0$.

Numerical results of errors $\|\bs u-\bs u_h\|_0$,  $\|\curl_h(\bs u-\bs u_h)\|_0$ and $|\curl_h(\bs u-\bs u_h)|_{1,h}$ with respect to $h$ for $k=0$ are shown in Table~\ref{table:errorenery}, from which we can see that they all achieve the optimal convergence rates numerically and agree with the theoretical error estimates in \eqref{eq:errorprioruh} and \eqref{eq:errorprioruhcurl}. It is also observed from Table~\ref{table:errorenery} that
$\|\curl_h(\bs u-\bs u_h)\|_0=O(h^2)$ numerically, which is one order higher than the theoretical order in~\eqref{eq:errorprioruhcurl}.
\begin{table}[htbp]
\caption{Errors $\|\bs u-\bs u_h\|_0$,  $\|\curl_h(\bs u-\bs u_h)\|_0$ and $|\curl_h(\bs u-\bs u_h)|_{1,h}$ for $k=0$ and different $h$.}\label{table:errorenery}
\begin{tabular}{lllllll}
\hline\noalign{\smallskip}
$h$ & $\|\bs u-\bs u_h\|_0$ & order & $\|\curl_h(\bs u-\bs u_h)\|_0$ & order & $|\curl_h(\bs u-\bs u_h)|_{1,h}$ & order \\
\noalign{\smallskip}\hline\noalign{\smallskip}
$2^{-1}$ & 1.025E+00 & $-$ & 1.050E+01 & $-$ & 1.076E+02 & $-$ \\
$2^{-2}$ & 9.687E$-$01 & 0.08 & 5.306E+00 & 0.98 & 9.099E+01 & 0.24 \\
$2^{-3}$ & 3.767E$-$01 & 1.36 & 1.618E+00 & 1.71 & 5.374E+01 & 0.76 \\
$2^{-4}$ & 1.640E$-$01 & 1.20 & 4.311E$-$01 & 1.91 & 2.820E+01 & 0.93 \\
$2^{-5}$ & 7.828E$-$02 & 1.07 & 1.097E$-$01 & 1.97 & 1.428E+01 & 0.98 \\
\noalign{\smallskip}\hline
\end{tabular}
\end{table}

\section*{Acknowledgement}
We greatly appreciate the anonymous reviewers for valuable suggestions and careful comments.

%


 \appendix
 \section{Regularity of the quad curl problem on convex domains}\label{sec:appendixA}

We will prove the regularity of problem \eqref{eq:quadcurlgrad} under the assumption $\bs f\in \bs H^{-1}(\div, \Omega)$. Similar regularity can be found in \cite[Theorem 3.5]{Zhang2018a} when $\bs f\in \boldsymbol{L}^{2}(\Omega;\mathbb R^3)$.
\begin{lemma}\label{lem:bigradcurlregularity}
Assume domain $\Omega$ is convex. Let $\bs  u\in\bs H_0(\grad\curl, \Omega)$ be the solution of problem \eqref{eq:quadcurlgrad} with the divergence-free right hand side $\bs f\in \bs H^{-1}(\div, \Omega)$, then
\begin{equation}\label{eq:appendixregularity}
\|\bs u\|_1+\|\curl \bs u\|_2\lesssim \|\bs f \|_{-1}.
\end{equation}
\end{lemma}
\begin{proof}
Due to the framework in \cite{ChenHuang2018}, the problem \eqref{eq:quadcurlgrad} can be equivalently decoupled into the following system:
find $\bs w\in \bs H_0(\curl, \Omega)$, $\lambda\in H_0^{1}(\Omega)$, $\boldsymbol{\phi}\in \boldsymbol H_0^1(\Omega; \mathbb{R}^3)$, $p\in L_0^{2}(\Omega)$, $\bs u\in \bs H_0(\curl, \Omega)$ and $\sigma\in H_0^{1}(\Omega)$ such that
\begin{align}
(\curl  \bs w, \curl  \bs v)+(\bs v,\nabla \lambda)&=\langle\bs f, \bs v\rangle  \quad\quad\quad\quad \forall~\bs v\in \bs H_0(\curl, \Omega),  \label{eq:bigradcurldecouple1}\\
(\bs w,\nabla \tau) &= 0 \quad\quad\quad\quad\quad\;\;\;\, \forall~ \tau\in H_0^{1}(\Omega),  \label{eq:bigradcurldecouple2}\\
(\boldsymbol\nabla\boldsymbol\phi, \boldsymbol\nabla\boldsymbol\psi) + (\div\boldsymbol\psi, p) & =(\curl  \bs w, \boldsymbol\psi) \quad\;\;\, \forall~\boldsymbol{\psi}\in  \boldsymbol H_0^1(\Omega; \mathbb{R}^3), \label{eq:bigradcurldecouple3}\\
(\div\boldsymbol\phi, q) &= 0 \quad\quad\quad\quad\quad\;\;\;\, \forall~ q\in L_0^{2}(\Omega),  \label{eq:bigradcurldecouple4}\\
(\curl \bs  u, \curl \bs  \chi)+(\bs \chi,\nabla\sigma)&= (\boldsymbol{\phi}, \curl \bs \chi) \quad\quad \forall~\bs \chi\in \bs H_0(\curl, \Omega),  \label{eq:bigradcurldecouple5} \\
(\bs u,\nabla \mu) &= 0 \quad\quad\quad\quad\quad\;\;\;\, \forall~ \mu\in H_0^{1}(\Omega).  \label{eq:bigradcurldecouple6}
\end{align}
Here $\langle\cdot, \cdot\rangle $ is the dual pair between $\bs H^{-1}(\div, \Omega)$ and $\bs H_0(\curl, \Omega)$.
Since $\bs w, \bs u\in\bs H_0(\curl, \Omega)\cap\bs H(\div, \Omega)$, we have $\bs w, \bs u\in\bs H^1(\Omega;\mathbb R^3)$ \cite[Section I.3.4]{GiraultRaviart1986} and
\[
\|\bs w\|_1\lesssim \|\curl\bs w\|_0\lesssim \|\bs f \|_{-1},
\]
\begin{equation}\label{eq:20200815-1}
\|\bs u\|_1\lesssim \|\curl\bs u\|_0\lesssim \|\bs\phi \|_{0}.
\end{equation}
By the regularity of the Stokes problem \eqref{eq:bigradcurldecouple3}-\eqref{eq:bigradcurldecouple4} \cite[Remark I.5.6]{GiraultRaviart1986},  we have
\begin{equation}\label{eq:20200815-2}
\|\bs\phi\|_2\lesssim \|\curl\bs w\|_0\lesssim \|\bs f \|_{-1}.
\end{equation}
Finally we conclude \eqref{eq:appendixregularity} from \eqref{eq:20200815-1}-\eqref{eq:20200815-2} and the fact $\bs\phi=\curl\bs u$.
\end{proof}

\bibliographystyle{abbrv}
\bibliography{./ncfemstokescomplex3d}

\begin{thebibliography}{10}

\bibitem{AmroucheBernardiDaugeGirault1998}
C.~Amrouche, C.~Bernardi, M.~Dauge, and V.~Girault.
\newblock Vector potentials in three-dimensional non-smooth domains.
\newblock {\em Math. Methods Appl. Sci.}, 21(9):823--864, 1998.

\bibitem{Arnold2018}
D.~N. Arnold.
\newblock {\em Finite element exterior calculus}.
\newblock Society for Industrial and Applied Mathematics (SIAM), Philadelphia,
  PA, 2018.

\bibitem{ArnoldFalkWinther2006}
D.~N. Arnold, R.~S. Falk, and R.~Winther.
\newblock Finite element exterior calculus, homological techniques, and
  applications.
\newblock {\em Acta Numer.}, 15:1--155, 2006.

\bibitem{AustinManteuffelMcCormick2004}
T.~M. Austin, T.~A. Manteuffel, and S.~McCormick.
\newblock A robust multilevel approach for minimizing {$\bold H({\rm
  div})$}-dominated functionals in an {$\bold H^1$}-conforming finite element
  space.
\newblock {\em Numer. Linear Algebra Appl.}, 11(2-3):115--140, 2004.

\bibitem{BeiraodaVeigaDassiVacca2020}
L.~Beir\~{a}o~da Veiga, F.~Dassi, and G.~Vacca.
\newblock The {S}tokes complex for virtual elements in three dimensions.
\newblock {\em Math. Models Methods Appl. Sci.}, 30(3):477--512, 2020.

\bibitem{BoffiBrezziFortin2013}
D.~Boffi, F.~Brezzi, and M.~Fortin.
\newblock {\em Mixed finite element methods and applications}.
\newblock Springer, Heidelberg, 2013.

\bibitem{BrennerScott2008}
S.~C. Brenner and L.~R. Scott.
\newblock {\em The mathematical theory of finite element methods}.
\newblock Springer, New York, third edition, 2008.

\bibitem{CakoniHaddar2007}
F.~Cakoni and H.~Haddar.
\newblock A variational approach for the solution of the electromagnetic
  interior transmission problem for anisotropic media.
\newblock {\em Inverse Probl. Imaging}, 1(3):443--456, 2007.

\bibitem{CaoChenHuang2020}
S.~Cao, L.~Chen, and X.~Huang.
\newblock Error analysis of a decoupled finite element method for quad-curl
  problems.
\newblock {\em J. Sci. Comput.}, 90(1):Paper No. 29, 25, 2022.

\bibitem{ChenHuHuang2018}
L.~Chen, J.~Hu, and X.~Huang.
\newblock Fast auxiliary space preconditioners for linear elasticity in mixed
  form.
\newblock {\em Math. Comp.}, 87(312):1601--1633, 2018.

\bibitem{ChenHuang2018}
L.~Chen and X.~Huang.
\newblock Decoupling of mixed methods based on generalized {H}elmholtz
  decompositions.
\newblock {\em SIAM J. Numer. Anal.}, 56(5):2796--2825, 2018.

\bibitem{ChenWuZhongZhou2018}
L.~Chen, Y.~Wu, L.~Zhong, and J.~Zhou.
\newblock Multi{G}rid preconditioners for mixed finite element methods of the
  vector {L}aplacian.
\newblock {\em J. Sci. Comput.}, 77(1):101--128, 2018.

\bibitem{ChristiansenWinther2008}
S.~H. Christiansen and R.~Winther.
\newblock Smoothed projections in finite element exterior calculus.
\newblock {\em Math. Comp.}, 77(262):813--829, 2008.

\bibitem{Ciarlet1978}
P.~G. Ciarlet.
\newblock {\em The finite element method for elliptic problems}.
\newblock North-Holland Publishing Co., Amsterdam, 1978.

\bibitem{CostabelMcIntosh2010}
M.~Costabel and A.~McIntosh.
\newblock On {B}ogovski\u\i\ and regularized {P}oincar\'e integral operators
  for de {R}ham complexes on {L}ipschitz domains.
\newblock {\em Math. Z.}, 265(2):297--320, 2010.

\bibitem{CrouzeixRaviart1973}
M.~Crouzeix and P.-A. Raviart.
\newblock Conforming and nonconforming finite element methods for solving the
  stationary {S}tokes equations. {I}.
\newblock {\em Rev. Fran\c caise Automat. Informat. Recherche Op\'erationnelle
  S\'er. Rouge}, 7(R-3):33--75, 1973.

\bibitem{ElmanSilvesterWathen2005}
H.~C. Elman, D.~J. Silvester, and A.~J. Wathen.
\newblock {\em Finite elements and fast iterative solvers: with applications in
  incompressible fluid dynamics}.
\newblock Oxford University Press, New York, 2005.

\bibitem{FalkMorley1990}
R.~S. Falk and M.~E. Morley.
\newblock Equivalence of finite element methods for problems in elasticity.
\newblock {\em SIAM J. Numer. Anal.}, 27(6):1486--1505, 1990.

\bibitem{FalkNeilan2013}
R.~S. Falk and M.~Neilan.
\newblock Stokes complexes and the construction of stable finite elements with
  pointwise mass conservation.
\newblock {\em SIAM J. Numer. Anal.}, 51(2):1308--1326, 2013.

\bibitem{GiraultRaviart1986}
V.~Girault and P.-A. Raviart.
\newblock {\em Finite element methods for {N}avier-{S}tokes equations}.
\newblock Springer-Verlag, Berlin, 1986.

\bibitem{GopalakrishnanDemkowicz2004}
J.~Gopalakrishnan and L.~F. Demkowicz.
\newblock Quasioptimality of some spectral mixed methods.
\newblock {\em J. Comput. Appl. Math.}, 167(1):163--182, 2004.

\bibitem{GuzmanLischkeNeilan2020}
J.~Guzm\'{a}n, A.~Lischke, and M.~Neilan.
\newblock Exact sequences on {P}owell-{S}abin splits.
\newblock {\em Calcolo}, 57(2):Paper No. 13, 25, 2020.

\bibitem{GuzmanNeilan2012}
J.~Guzm{\'a}n and M.~Neilan.
\newblock A family of nonconforming elements for the {B}rinkman problem.
\newblock {\em IMA J. Numer. Anal.}, 32(4):1484--1508, 2012.

\bibitem{GuzmanNeilan2014}
J.~Guzm{\'a}n and M.~Neilan.
\newblock Conforming and divergence-free {S}tokes elements on general
  triangular meshes.
\newblock {\em Math. Comp.}, 83(285):15--36, 2014.

\bibitem{Hiptmair1999}
R.~Hiptmair.
\newblock Canonical construction of finite elements.
\newblock {\em Math. Comp.}, 68(228):1325--1346, 1999.

\bibitem{HiptmairXu2007}
R.~Hiptmair and J.~Xu.
\newblock Nodal auxiliary space preconditioning in {${\bf H}({\bf curl})$} and
  {${\bf H}({\rm div})$} spaces.
\newblock {\em SIAM J. Numer. Anal.}, 45(6):2483--2509, 2007.

\bibitem{HuZhangZhang2020}
K.~Hu, Q.~Zhang, and Z.~Zhang.
\newblock Simple curl-curl-conforming finite elements in two dimensions.
\newblock {\em SIAM J. Sci. Comput.}, 42(6):A3859--A3877, 2020.

\bibitem{HuZhangZhang2020curcurl3d}
K.~Hu, Q.~Zhang, and Z.~Zhang.
\newblock A family of finite element {S}tokes complexes in three dimensions.
\newblock {\em SIAM J. Numer. Anal.}, 60(1):222--243, 2022.

\bibitem{Huang2010}
X.~Huang.
\newblock {\em New finite element methods and efficient algorithms for fourth
  order elliptic equations}.
\newblock PhD thesis, Shanghai Jiao Tong University, 2010.

\bibitem{JohnLinkeMerdonNeilanEtAl2017}
V.~John, A.~Linke, C.~Merdon, M.~Neilan, and L.~G. Rebholz.
\newblock On the divergence constraint in mixed finite element methods for
  incompressible flows.
\newblock {\em SIAM Rev.}, 59(3):492--544, 2017.

\bibitem{Lee2010}
Y.-J. Lee.
\newblock Uniform stability analysis of {A}ustin, {M}anteuffel and
  {M}c{C}ormick finite elements and fast and robust iterative methods for the
  {S}tokes-like equations.
\newblock {\em Numer. Linear Algebra Appl.}, 17(1):109--138, 2010.

\bibitem{MardalTaiWinther2002}
K.~A. Mardal, X.-C. Tai, and R.~Winther.
\newblock A robust finite element method for {D}arcy-{S}tokes flow.
\newblock {\em SIAM J. Numer. Anal.}, 40(5):1605--1631, 2002.

\bibitem{Nedelec1980}
J.-C. N{\'e}d{\'e}lec.
\newblock Mixed finite elements in {${\bf R}^{3}$}.
\newblock {\em Numer. Math.}, 35(3):315--341, 1980.

\bibitem{Nedelec1986}
J.-C. N{\'e}d{\'e}lec.
\newblock A new family of mixed finite elements in {${\bf R}^3$}.
\newblock {\em Numer. Math.}, 50(1):57--81, 1986.

\bibitem{Neilan2015}
M.~Neilan.
\newblock Discrete and conforming smooth de {R}ham complexes in three
  dimensions.
\newblock {\em Math. Comp.}, 84(295):2059--2081, 2015.

\bibitem{RaviartThomas1977}
P.-A. Raviart and J.~M. Thomas.
\newblock A mixed finite element method for 2nd order elliptic problems.
\newblock In {\em Mathematical aspects of finite element methods ({P}roc.
  {C}onf., {C}onsiglio {N}az. delle {R}icerche ({C}.{N}.{R}.), {R}ome, 1975)},
  pages 292--315. Lecture Notes in Math., Vol. 606. Springer, Berlin, 1977.

\bibitem{ScottZhang1990}
L.~R. Scott and S.~Zhang.
\newblock Finite element interpolation of nonsmooth functions satisfying
  boundary conditions.
\newblock {\em Math. Comp.}, 54(190):483--493, 1990.

\bibitem{TaiWinther2006}
X.-C. Tai and R.~Winther.
\newblock A discrete de {R}ham complex with enhanced smoothness.
\newblock {\em Calcolo}, 43(4):287--306, 2006.

\bibitem{WangXu2006}
M.~Wang and J.~Xu.
\newblock The {M}orley element for fourth order elliptic equations in any
  dimensions.
\newblock {\em Numer. Math.}, 103(1):155--169, 2006.

\bibitem{ZhangWangZhang2019}
Q.~Zhang, L.~Wang, and Z.~Zhang.
\newblock {$H({\rm curl}^2)$}-conforming finite elements in 2 dimensions and
  applications to the quad-curl problem.
\newblock {\em SIAM J. Sci. Comput.}, 41(3):A1527--A1547, 2019.

\bibitem{ZhangZhang2020}
Q.~Zhang and Z.~Zhang.
\newblock A family of curl-curl conforming finite elements on tetrahedral
  meshes.
\newblock {\em CSIAM Transactions on Applied Mathematics}, 1(4):639--663, 2020.

\bibitem{Zhang2016}
S.~Zhang.
\newblock Stable finite element pair for {S}tokes problem and discrete {S}tokes
  complex on quadrilateral grids.
\newblock {\em Numer. Math.}, 133(2):371--408, 2016.

\bibitem{Zhang2018a}
S.~Zhang.
\newblock Mixed schemes for quad-curl equations.
\newblock {\em ESAIM Math. Model. Numer. Anal.}, 52(1):147--161, 2018.

\bibitem{ZhengHuXu2011}
B.~Zheng, Q.~Hu, and J.~Xu.
\newblock A nonconforming finite element method for fourth order curl equations
  in {$\Bbb{R}^{3}$}.
\newblock {\em Math. Comp.}, 80(276):1871--1886, 2011.

\end{thebibliography}
\end{document}